\documentclass[11pt, twoside]{amsart}
\usepackage{inputenc}
\scrollmode 
\usepackage{inputenc}
\usepackage{amssymb}
\usepackage[matrix,arrow,curve]{xy}
\usepackage{amsmath}
\usepackage[french,english]{babel}
\usepackage[dvips]{xcolor}

\newtheorem{thm}{Theorem}[section]
\newtheorem{prop}[thm]{Proposition}
\newtheorem{lemma}[thm]{Lemma}
\newtheorem{cor}[thm]{Corollary}

\theoremstyle{definition}
\newtheorem{definition}[thm]{Definition}
\newtheorem{def-thm}[thm]{Definition-Theorem}
\newtheorem{rem}[thm]{Remark}

\numberwithin{equation}{section}

\def\diag{\operatorname{diag}}

\def\dual{\operatorname{dual}}

\def\ev{\operatorname{ev}}

\def\Id{\operatorname{Id}}

\def\Norm{\operatorname{Norm}}

\def\PGL{\operatorname{PGL}}

\def\red{\operatorname{red}}
\def\sgn{\operatorname{sgn}}
\def\sl{\operatorname{sl}}
\def\SL{\operatorname{SL}}

\def\th{\operatorname{th}}

\begin{document}

\title
{Cluster $\mathcal X$-varieties for dual Poisson-Lie groups II}

\author{Renaud Brahami}
\address{Section of Mathematics, University of Geneva, 2-4 rue du Li\`{e}vre, c.p. 64, 1211
Gen\`{e}ve 4, Switzerland
}
\email{Renaud.Brahami@unige.ch}

\maketitle

\begin{abstract}
In the prequel of this paper, we have associated
a family of cluster $\mathcal X$-varieties to the dual
Poisson-Lie group $G^*\subset(G,\pi_*)$ of $(G,\pi_G)$ when $(G,\pi_G)$
is a complex semi-simple Lie group of adjoint type, given with the standard Poisson
structure $\pi_G$ and $\pi_*$ is the "dual" Poisson structure
defined by the Semenov-Tian-Shansky Poisson bracket on $G$.
We describe here the cluster combinatorics involved into the Artin group action on $G^*$ given
by the De-Concini-Kac-Procesi Poisson automorphisms.
\end{abstract}

\section{Introduction}
A cluster $\mathcal X$-variety is a Poisson variety obtained by gluing a set of algebraic tori
along some specific bi-rational isomorphisms called ($\mathcal X$-)mutations, which are strongly
related \cite{FGdilog} to the mutations of the well-known cluster algebra of Fomin and Zelevinsky
introduced in \cite{FZ1}.
Each torus is given a log-canonical Poisson structure, that is a set of
coordinates $x_i$ and a skew-symmetric matrix $\widehat{\varepsilon}$, with
generic integer values, such that the equality $\{x_i,x_j\}=\widehat{\varepsilon}_{ij}x_ix_j$
is satisfied.
Because mutations are Poisson maps relative to these log-canonical Poisson
structures, cluster $\mathcal X$-varieties are naturally given a kind of
Darboux coordinates. Therefore, in the same way that the cluster algebra machinery
can be used to described coordinates ring of affine varieties related to semisimple
Lie groups \cite{BFZbruhat}, \cite{S}, we can use cluster $\mathcal X$-varieties to
study their Poisson geometry \cite{FGcluster}, \cite{GSV}, \cite{RB}.

When $G$ is a {real} split semisimple Lie group, with trivial center,
given with the Sklyanin Poisson structure associated with the standard
$r$-matrix of the Belavin-Drinfeld classification (this makes $G$ a
Poisson Lie group), Fock and Goncharov have, in the paper \cite{FGcluster}, 
constructed canonical Poisson birational maps of cluster
$\mathcal X$-varieties into $G$ (one map for each seed $\mathcal X$-torus
associated with a double reduced word of the Weyl group $W$ of $G$); this construction
provides for $G$ a natural set of rational canonical coordinates. Canonical maps
associated with different double reduced words are given by a composition of
mutations simply related to the composition of generalized $d$-moves
linking double reduced words.

Now, let us recall that any Poisson-Lie group structure on any Lie group $G$ provides
its Lie algebra $\mathfrak g$ with a
bialgebra structure, and thus its dual $\mathfrak g^{*}$ acquires a bialgebra
structure as well. Hence Poisson-Lie groups always come by pairs, and the Lie group
associated to $\mathfrak g^*$ is provided
a Poisson-Lie group structure and is called the dual of $G$. In particular, the dual of
a {complex} semisimple Poisson-Lie group $(G,\pi_G)$ of adjoint type, still equipped
with the Sklyanin Poisson structure associated with the standard
$r$-matrix, can be identified
with a subgroup in the direct product of two opposite Borel subgroups of $G$.
This group may also be mapped onto a dense open subvariety in $G$. The induced Poisson
structure $\pi_*$ then extends smoothly to the entire group $G$ and is given
by a simple explicit formula \cite{STS}. In \cite{dCKP}, De Concini, Kac and
Procesi described an Artin group action on {$G^*\subset(G,\pi_*)$} by Poisson
transformations. It turns out that this action is the semiclassical analog of the Lusztig
automorphisms on the universal quantized enveloping algebra $\mathcal{U}_q(\mathfrak{g})$
of the Lie algebra $\mathfrak{g}$.

In the prequel to this paper \cite{RB}, we have associated a family of cluster
$\mathcal X$-varieties $\{{\mathcal X}_w\mid w\in W\}$ to the dual
Poisson-Lie group~$G^*\subset(G,\pi_*)$. 
The underlying combinatorics is based on a factorization
of the Fomin-Zelevinsky twist maps into mutations and other new Poisson
birational isomorphisms on seed $\mathcal X$-tori called {tropical
mutations}, associated to an enrichment of the combinatorics on double words
of the Weyl group $W$ of~$G$. A double word $\mathbf i$ of $W$ being a word
of $W\times W$, this enrichment is in fact based on a switch $i\mapsto \overline{i}$
of the two copies of $W$, acting on the first or the last letter of a double word $\mathbf i$.
Such new moves are called $\tau$-moves, and tropical mutations are then Poisson birational
isomorphisms between the related seed $\mathcal X$-tori. Finally, twisted evaluations, strongly
influenced by the morphisms of Evens and Lu \cite{EL}, relate the cluster
$\mathcal X$-varieties ${\mathcal X}_w$ to $(G,\pi_*)$.

In the present paper, we use the previous combinatorics to get an action of the Artin group associated
to the root system of the Lie algebra $\mathfrak g$ by Poisson automorphisms on any
seed $\mathcal X$-torus of any cluster $\mathcal X$-variety $\mathcal X_w$.
Composing these automorphisms with some of the twisted evaluations of \cite{RB}, we rediscover,
as a particular case, the Artin group action on $G^*\subset(G,\pi_*)$
generated by the De-Concini-Kac-Procesi automorphisms.

Here is the organization of the paper: Section \ref{section:prelimanaries}
collects preliminaries, Section \ref{section:cluster} recalls how to
associate cluster $\mathcal X$-varieties to the Poisson manifolds $(G,\pi_G)$ and $(G,\pi_*)$, Section
\ref{section:Artin} deals with the cluster combinatorics of the Artin group action
on $G^*\subset(G,\pi_*)$ leading to the De-Concini-Kac-Procesi Poisson automorphisms,
and Section \ref{section:SL2} details our construction when $G$ is of type $A_1$.

\section{Preliminaries}
\label{section:prelimanaries}
\subsection{Lie setting}
Let $\mathfrak g$ be a complex semi-simple Lie algebra  of rank $\ell$,
$A$ its Cartan matrix, and $G$ its connected Lie group of adjoint type. Fix a
Borel subgroup $B\subset G$, let $B_-$ be the opposite Borel subgroup, $H=B\cap B_-$
the associated Cartan subgroup and $\mathfrak{h}\subset \mathfrak{g}$ the  corresponding
Cartan subalgebra. In the following, we will denote
$[1,\ell]=\{1,\dots,\ell\}$.
Let
$\alpha_1,\dots,\alpha_l$ be the simple roots of $\mathfrak g$, and let
$\omega_1, \omega_2, \ldots, \omega_l\in{\mathfrak h}^*$ be the corresponding
{fundamental weights}.
For every $i\in[1,l]$, let $(e_i,f_i,h_i)$ be the Chevalley generators
of $\mathfrak g$; they generate a Lie subalgebra $\mathfrak{g}_{\alpha_i}$
of $\mathfrak{g}$. In particular, we have $\omega_j(h_k) = \delta_{jk}$
for every $j, k \in[1,l]$. Let us recall that the \emph{weight lattice}
$\Lambda$ is the set of all weights $\gamma \in \mathfrak{h}^*$
such that $\gamma (h_i) \in \mathbb{Z}$ for all $i$.
Every weight $\gamma \in\Lambda$ gives rise to a multiplicative character
$a \mapsto a^\gamma$ of the Cartan subgroup $H$, given by
$\exp (h)^\gamma = e^{\gamma (h)}$, with $h \in \mathfrak{h}$.
We introduce a new basis on $\mathfrak h$ putting
\begin{equation}\label{equ:hrel}
h^i:=\sum{(A^{-1})}_{ij}\ h_j\ .
\end{equation}
Let  $D=\diag(d_1,\dots,d_l)$ be the diagonal matrix symmetrizing the
Cartan matrix; we put $\widehat{a}_{ij}=d_ia_{ij}=a_{ij}d_j$.
For every $x\in\mathbb C$ and $i\in[1,\ell]$, we define the group elements
$E^i=\exp(e_i)$, $F^i=\exp(f_i)$ and $H^{i}(x)=\exp(\log(x)h^{i})$ related
to the generators $e_i$, $f_i$ and $h^i$ of $\mathfrak{g}$.
Because of the relation (\ref{equ:hrel}), the
canonical inclusions $\varphi_i: SL(2,\mathbb C)\hookrightarrow G$ satisfy
$$
\begin{array}{cc}
\varphi_i\left(
\begin{array}{cc}
1 & x\\
0 & 1
\end{array}
\right) =H^i(x)E^iH^i(x^{-1}),
& \varphi_i\left(
\begin{array}{cc}
1 & 0\\
x & 1
\end{array}
\right) =H^i(x^{-1})F^iH^i(x)\ .
\end{array}
$$

We denote by $W$ the Weyl group of $G$.
As an abstract group, $W$ is a finite Coxeter group of rank $\ell$ generated by
the set of \emph{simple reflections} $S=\{s_1, \dots, s_\ell\}$; it acts on
$\mathfrak h^*$, $\mathfrak h$ and the Cartan subgroup $H$ by
\begin{equation}\label{equ:actionW}
\begin{array}{ccccc}
s_i (\gamma) = \gamma - \gamma (\alpha_i^\vee) \alpha_i&,&
s_i (h) = h - \alpha_i (h) \alpha_i^\vee&\mbox{and}&
a^{w(\gamma)}=(\widehat w^{-1}a\widehat{w})^{\gamma}
\end{array}
\end{equation}
for every $\gamma \in \mathfrak h^*$, $h \in \mathfrak h$,
$w\in W$ and $a\in H$.
Recall now that a \emph{reduced word}
for $w\in W$ is an expression for $w$ in the generators belonging to $S$, which is
minimal in length among all such expressions for $w$. Let us denote $\ell(w)$
this minimal length and $R(w)$ the set of reduced words associated to $w$. As
usual, the notation $w_0$ will refer to the longest element of $W$. Let us also recall
that $W$ can also be seen as the subgroup $\Norm_G(H)/H$ of $G$.
Thus, to every simple reflection
$s_i\in W$ we associate the representative
$$\widehat{s_i}=\phi_i\left(
\begin{array}{rr}
0&-1\\
1&0
\end{array}
\right)\ .
$$
We can choose  representatives in $G$ for every element of $W$ by setting
$\widehat{w_1w_2}=\widehat{w_1}\widehat{w_2}$ for every $w_1,w_2\in W$
as long as $\ell(w_1)+\ell(w_2)=\ell(w_1w_2)$.

\subsection{Poisson setting}
Let us  recall that  the so-called \emph{standard classical $r$-matrix} is given by
$r=\sum e_{\alpha}\wedge f_{\alpha}\in\mathfrak{g}\wedge\mathfrak{g}$ {
(the summation is done {over all} positive roots $\alpha$)}.
Let $\langle\ ,\ \rangle$ be the Killing form on $\mathfrak{g}$.
For every $x\in\mathfrak{g}$, $X\in G$ and $f\in{\mathcal F}(G)$
the left and right gradients are defined respectively by
$$
\begin{array}{ccc}
\langle \nabla f(X), x\rangle =\frac{d}{dt}\mid_{t=0}f(e^{tx}X)& \text{and}&
\langle \nabla 'f(X), x\rangle =\frac{d}{dt}\mid_{t=0}f(Xe^{tx})\ .
\end{array}$$
If $f,g\in{\mathcal F}(G)$, let $\pi_G$ be the following Poisson structure
on $G$ given by the Sklyanin bracket which transforms $G$ into a Poisson-Lie group.
$$\{f,g\}_G=\frac{1}{2}(\langle \nabla f\otimes\nabla g,r\rangle
-\langle \nabla ' f\otimes\nabla ' g,r\rangle )\ .$$
Denote $r_{\pm}=r\pm t$, where $t$ is the Casimir element of $\mathfrak{g}$:
$$t=\displaystyle\frac{1}{2}\sum_{i=1}^{\ell}(h_i\otimes h^i+e_i\otimes f_i+f_i\otimes e_i)
\in{\mathfrak g}\otimes{\mathfrak g}\ .$$

\begin{prop}[\cite{STS}]\label{prop:STSPoisson} Let us equip $G$ with the Poisson structure $\pi_*$ given by $$
\{f,g\}_*=\frac{1}{2}(\langle \nabla f\otimes\nabla g,r\rangle
+\langle \nabla ' f\otimes\nabla ' g,r\rangle )
-\langle \nabla f\otimes\nabla' g,r_+\rangle -\langle \nabla ' f\otimes\nabla g,r_-\rangle\ .
$$
There exists a canonical map of the dual group $G^*$ onto the dense open cell
$BB_-\subset G$ which is a covering of Poisson manifolds.
\end{prop}

To any $u,v\in W$ we associate the \emph{double Bruhat cell}
$G^{u,v}=B\widehat{u}B\cap B_-\widehat{v}B_-$; we have
$$G=\displaystyle\bigcup_{u,v\in W}G^{u,v}\ .$$

\begin{prop}[\cite{KZ}]For every $u,v\in W$, the double Bruhat cells $G^{u,v}$
are the $H$-orbits, by the right (or left) action, of the symplectic leaves
of $(G,\pi_G)$.
\end{prop}

Now, following \cite{EL}, let us now give a Poisson stratification for $(G,\pi_*)$.
Two elements $g_1, g_2 \in G$ are said to be in the same \emph{Steinberg fiber} if
$f(g_1) = f(g_2)$ for every regular function $f$ on $G$ that is invariant under
conjugation. For $t \in H$, let $F_t$ be the Steinberg fiber containing $t$.
By the Jordan decomposition of elements in $G$, every Steinberg fiber is of the
form $F_t$ for some $t \in H$. Moreover, the equality $F_{t'} = F_t$ is satisfied
if and only if there exists $w\in W$ such that $t'=w(t)$, where $W$ acts on $H$ by
the formula (\ref{equ:actionW}).
The group $G$ has therefore the decompositions
\begin{equation}\label{equ:decG^*}
\begin{array}{cccc}
G=\displaystyle\bigcup_{t\in H,w\in W}F_{t,w}
=\displaystyle\bigsqcup_{t\in H\backslash W,w\in W}F_{t,w}
&&\mbox{where}&F_{t,w}:=B\widehat{w}B_-\bigcap F_t\ .
\end{array}
\end{equation}

\begin{prop}\cite[Proposition 3.3]{EL} Every $F_{t,v}$ is a finite union of
$H$-orbits, {with respect to} the conjugation action, of the symplectic leaves
of $(G,\pi_*)$.
\end{prop}

\subsection{Cluster $\mathcal X$-variety setting} We recall here the definitions
introduced by Fock and Goncharov underlying cluster $\mathcal X$-varieties
(see, for example, \cite{FGdilog} for more details).
A \emph{seed} ${\mathbf I}$ is a quadruple $(I, I_0, \varepsilon, d)$ where
\begin{itemize}
\item $I$ is a finite set;

\item $I_0\subset I$;

\item $\varepsilon$ is a matrix $\varepsilon_{ij}$, $i,j\in I$, such
that $\varepsilon_{ij}\in \mathbb{Z}$, unless $i,j\in I_0$;

\item $d=\{d_i\}$, $i \in I$, is a subset of positive integers such that
the matrix $\widehat{\varepsilon}_{ij}=\varepsilon_{ij}d_j$ is skew-symmetric.
\end{itemize}
For every real number $x\in\mathbb R$, let us denote
$[x]_+= \max(x,0)$ and
$$
\sgn(x)=
\left\{\begin{array}{cc}
-1 & \text{if $x<0$ ;}\\
0  & \text{if $x=0$ ;}\\
 1 & \text{if $x>0$ .}
\end{array}
\right.
$$
Let $\mathbf{I}=(I,I_0,\varepsilon,d)$, $\mathbf {I'}=(I',{I'}_0,{\varepsilon}',d')$
be two seeds, and $k\in I\backslash I_0$. A $\emph{mutation in the direction k}$ is a map
$\mu_k:I\longrightarrow I'$ satisfying the following conditions:
\begin{itemize}
\item
$\mu_k(I_0)={I'}_0$;
\item
${d'}_{\mu_k(i)}=d_i$;
\item
${\varepsilon^{'}}_{\mu_k(i)\mu_k(j)}=\left\{ \begin{array}{lll}
-\varepsilon_{ij}&   \mbox{if}\ i=k\ \mbox{or}\ j=k\ ;\\
\varepsilon_{ij}+\sgn(\varepsilon_{ik})[\varepsilon_{ik}\varepsilon_{kj}]_+& \mbox{if}\  i,j\neq k\ .
\end{array}\right.$
\end{itemize}
A $\emph{symmetry}$ on a seed $\mathbf{I}=(I,I_0,\varepsilon,d)$ is an
automorphism $\sigma$ of the set $I$ preserving the subset $I_0$, the matrix
$\varepsilon$, and the numbers $d_i$. That is to say:
\begin{itemize}
\item
$\sigma (I_0)=I_0$;
\item
$d_{\sigma (i)}=d_i$;
\item
${\varepsilon}_{\sigma (i)\sigma (j)}={\varepsilon}_{ij}.$
\end{itemize}
Let $|I|$ be the cardinal of every finite set $I$ and
${\mathbb C}_{\neq 0}$ be the set of non-zero complex numbers.
The \emph{seed} $\mathcal X$-\emph{torus}
$\mathcal{X}_{\mathbf I}$ associated
to a seed $\mathbf I$ is the torus $(\mathbb C_{\neq 0})^{| I|}$ with
the Poisson bracket
$$\{x_i,x_j\}=\widehat{\varepsilon}_{ij}x_ix_j\ ,$$
where $\{x_i| i\in I\}$ are the standard coordinates on the factors.
Symmetries and mutations on seeds induce involutive maps between the
corresponding seed $\mathcal X$-tori, which are denoted by the same symbols
${\mu}_k$ and $\sigma$, and given by
\begin{itemize}
\item $
\sigma^*x_{\sigma(i)}=x_i$
\item
$\mu_k^*x_{\mu_k(i)}=\left\{ \begin{array}{lll}
{x_k}^{-1}&  \mbox{if}\ i=k\ ;\\
x_ix_k^{[\varepsilon_{ik}]_+}(1+x_k)^{-\varepsilon_{ik}}& \mbox{if}\ i\not=k\ .
\end{array}\right.
$
\end{itemize}
Finally, a \emph{cluster transformation} linking two seeds
(and two seed $\mathcal X$-tori) is a composition of symmetries and mutations, and the
$\emph{cluster}$ ${\mathcal X}$-$\emph{variety}$ ${\mathcal X}_{|\mathbf I|}$ associated
to a seed $\mathbf I$
is obtained by taking every seed $\mathcal X$-tori obtained from
${\mathcal X}_{\mathbf I}$ by cluster transformations, and gluing them via
the previous bi-rational isomorphisms.

\section{{Cluster $\mathcal X$-varieties} related to $(G,\pi_G)$ and $(G,\pi_*)$}
\label{section:cluster}
\subsection{{Cluster $\mathcal X$-varieties} related to $(G,\pi_G)$}
\label{section:clusterpiG}
We briefly recall the construction of \cite{FGcluster}.
A (reduced) word of $W\times W$ is called a $\emph{double (reduced) word}$.
To avoid confusions, we denote ${\overline 1},\dots,{\overline \ell}$ the
indices of the reflections associated to the first copy of the Weyl group $W$ and
$1,\dots,\ell$ the indices of the reflections associated to the second copy of $W$.
A double reduced word of $(u,v)$ is a then shuffle of a reduced word of $u$,
written in the alphabet $[{\overline 1},{\overline \ell}]$,
and of a reduced word of $v$, written in the alphabet $[1,\ell]$.
We denote $R(u,v)$ the set of double reduced words of $(u,v)$, and $\mathbf 1\in R(1,1)$
the unity of $W\times W$.

For every $u,v\in W$, ${\mathbf i}\in R(u,v)$ and $j\in[1,\ell]$,
let $N^{j}({\mathbf i})$ be the number of times the letter $j$ or
$\overline j$ appear in $\mathbf i$. {Let $I(\mathbf i)$ (resp. $I_0^{\mathfrak{R}}(\mathbf i)$
and $I_0(\mathbf i)$) be the set of all ordered pairs $\binom{j}{k}$ such that
$j\in[1,\ell]$, and $0\leq k \leq N^{j}({\mathbf i})$ (resp. $k=N^{j}({\mathbf i})$
and $k\in\{0,N^{j}({\mathbf i})\}$).
Let us start to describe the trivial seed $(I(\mathbf 1),I_0(\mathbf 1),\varepsilon({\mathbf 1}),
d({\mathbf 1}))$. It is defined by the zero square matrix $\varepsilon({\mathbf 1})$ of size $\ell$ and
$\ell$-vectors $d(\mathbf 1)_{\binom j{0}}={d_j}$.

Now, let us describe the elementary
seeds $(I(i),I_0(i),\varepsilon(i),d(i))$ and $(I(\overline{i}),I_0(\overline{i}),
\varepsilon(\overline{i}),d(\overline{i}))$ for every $i\in[1,\ell]$ and $\mathbf i\in\{i,\overline i\}$. Let
$\varepsilon(i)$, $\varepsilon(\overline i)$ be the square matrices of size $\ell+1$, with entries labeled
by the elements of $I(\mathbf i)$ and given by
\begin{equation}\label{equ:exeps}
\begin{array}{cccccc}
\varepsilon(i)_{\binom i{1}\binom j{0}}=\displaystyle\frac{a_{ij}}{2}
=-\varepsilon(i)_{\binom i{0}\binom j{0}},
&&
\varepsilon(\overline i)_{\binom i{1}\binom j{0}}=-\displaystyle\frac{a_{ij}}{2}
=-\varepsilon(\overline i)_{\binom i{0}\binom j{0}},
\end{array}
\end{equation}
and zero otherwise.
In a similar way, let $d(i)$ and $d(\bar{i})$ be the $(\ell+1)$-vectors with components
labeled by the elements of $I(\mathbf{i})$,
$d(i)_{\binom j{k}}={d_j}=d(\overline i)_{\binom j{k}}$.

For every $\mathbf i\in\{\mathbf 1,i,\overline i\}$ we then
denote ${\mathcal X}_{\mathbf i}$
the seed $\mathcal X$-torus associated to the seed
$(I(\mathbf i),I_0(\mathbf i),\varepsilon({\mathbf i}),d({\mathbf i}))$ and
$\ev_{\mathbf i}:{\mathcal X}_{\mathbf i}\rightarrow G$ the related
evaluation map

\begin{align*}
\ev_{\mathbf 1}&\colon \mathbb C_{\neq 0}^{\ell} \longrightarrow G\hspace{-1,5cm}&&\colon\hspace{-1,5cm}&
\left(x_{\binom{1}{0}},\dots,x_{\binom{i}{0}},\dots,x_{\binom{l}{0}}\right)& \longmapsto \displaystyle\prod_{j}H^j(x_{\binom{j}{0}}) ,\\
\ev_{i}&\colon \mathbb C_{\neq 0}^{\ell+1} \longrightarrow G\hspace{+1,5cm}&&\colon\hspace{-1,5cm}&
(x_{\binom{1}{0}},\dots,x_{\binom{i}{0}},x_{\binom{i}{1}},x_{\binom{i+1}{0}},\dots,x_{\binom{l}{0}}) & \longmapsto \displaystyle\prod_{j}H^j(x_{\binom{j}{0}})E^iH^i(x_{\binom{i}{1}})\ ,\\
\ev_{\overline{i}}&\colon \mathbb C_{\neq 0}^{\ell+1} \longrightarrow G\hspace{-0cm}&&\colon\hspace{-1,5cm} &
(x_{\binom{1}{0}},\dots,x_{\binom{i}{0}},x_{\binom{i}{1}},x_{\binom{i+1}{0}},\dots,x_{\binom{l}{0}}) & \longmapsto \displaystyle\prod_{j}H^j(x_{\binom{j}{0}})F^iH^i(x_{\binom{i}{1}})\ .
\end{align*}

To associate seeds to longer words we proceed inductively. We associate
to a product $\mathbf i\mathbf j$ of double words $\mathbf i$ and $\mathbf j$
an \emph{amalgamated seed}
$(I(\mathbf i\mathbf j),I_0(\mathbf i\mathbf j),\varepsilon({\mathbf i
\mathbf j}),d({\mathbf i\mathbf j}))$
in the following way. The elements of the set $d({\mathbf i\mathbf j})$ are
{defined by $d({\mathbf i\mathbf j})_{\binom j{k}}={d_j}$},
and the matrix $\varepsilon({\mathbf i\mathbf j})$ is given by

\begin{equation}\label{equ:amaleps}
\varepsilon({\mathbf i\mathbf j})_{\binom i{k}\binom j{l}}=\left\{
\begin{array}{llll}
{\varepsilon}({\mathbf{i}})_{\binom i{k}\binom j{l}}
&\mbox{if}\ k<N^i(\mathbf{i})\mbox{ and }l<N^j(\mathbf{i});\\
{\varepsilon}({\mathbf{i}})_{\binom i{k}\binom j{l}}
+ {\varepsilon}({\mathbf{j}})_{\binom i{0}\binom j{0}}
&\mbox{if}\ k=N^i(\mathbf{i})\mbox{ and }l=N^j(\mathbf{i});\\
{\varepsilon}({\mathbf{j}})_{\binom i{k-N^i(\mathbf{i})}
\binom j{l-N^j(\mathbf{i})}}&\mbox{if}\ k>N^i(\mathbf{i})\mbox{ and }l>N^j(\mathbf{i});\\
0&\mbox{otherwise}.
\end{array}
\right.
\end{equation}
This operation induces
a homomorphism $\mathfrak{m}:{\mathcal X}_{\mathbf i}\times{\mathcal X}_{\mathbf j}
\rightarrow{\mathcal X}_{\mathbf i\mathbf j}$ between the corresponding seed $\mathcal X$-tori
called \emph{amalgamation} and given by
$$
\begin{array}{ll}
\mathfrak{m}^*{z}_{\binom{i}{k}}  =  \left\{
\begin{array}{lll}
x_{\binom{i}{k}}&\mbox{if}\ 0\leq k<N^i(\mathbf{i});\\
x_{\binom{i}{k}}y_{\binom{i}{0}}&\mbox{if}\ k=N^i(\mathbf{i});\\
y_{\binom{i}{k-N^i(\mathbf i)}}&\mbox{if}\ N^i(\mathbf{i})<k\leq N^i(\mathbf{i})+N^i(\mathbf{j}),
\end{array}
\right.\\
\end{array}
$$
where $x_i$, $y_j$ and $z_k$ denote respectively the coordinates functions on
${\mathcal X}_{\mathbf i}$, ${\mathcal X}_{\mathbf j}$, and
${\mathcal X}_{\mathbf i\mathbf j}$. Notice that this amalgamated
product is associative. Now, let $\mathbf i=i_1\dots i_k$ be a double word,
${\mathcal X}_{\mathbf i}$ be the seed $\mathcal X$-torus given by the associated amalgamation
$\mathfrak{m}:{\mathcal X}_{{i_1}}\times \dots\times{\mathcal X}_{{i_k}}\rightarrow
{\mathcal X}_{\mathbf i}$, and $\mathbf z\in{\mathcal X}_{\mathbf i}$ be an amalgamation
$\mathfrak{m}(\mathbf{x_1},\dots,\mathbf{x_k})$ of $\mathbf{x_1}\in{\mathcal X}_{{i_1}}$,
..., $\mathbf{x_k}\in{\mathcal X}_{{i_k}}$.
We define the evaluation map
$$\ev_{\mathbf i}:{\mathcal X}_{\mathbf i}\rightarrow G:
\mathbf z\mapsto \ev_{i_1}(\mathbf{x_1})\dots \ev_{i_k}(\mathbf{x_k})\ .$$

\begin{thm}[\cite{FGcluster}]\label{thm:ev}\label{thm:evG} For any $u,v\in W$ and $\mathbf i\in R(u,v)$ the map
$\ev_{\mathbf i}:{\mathcal X}_{\mathbf i}\rightarrow (G^{u,v},\pi_G)$ is a
Poisson birational isomorphism onto a Zariski open set of the double Bruhat cell
$G^{u,v}$.
\end{thm}

Let  $\mathbf i$ be a reduced word.
Following \cite{BZtensor}, we call a \emph{$d$-move}
a transformation of $\mathbf i$ that replaces $d$ consecutive entries
$i,j,i,j, \ldots$ by $j,i,j,i,\ldots$,
for some $i$ and $j$ such that $d$ is the order of $s_i s_j$, that is:
if $a_{ij}a_{ji} = 0$ (resp.\ $1,2,3$), then $d=2$ (resp.\ $3,4,6$).
By the Tits theorem, every two reduced words which represent the same element of a
Coxeter group are related by a sequence of $d$-moves.
Next, let us  say that a letter $i$ of $\mathbf i$ is \emph{positive}
if $i\in[1,\ell]$ and \emph{negative} if
$i\in[\overline 1,\overline \ell]$. Considering the group $W \times W$,
we conclude that every two double reduced words $\mathbf{i}, \mathbf{j} \in
R(u,v)$ can be obtained from each other by a sequence of
\emph{generalized $d$-moves}, i.e. \emph{positive} $d$-moves for the alphabet $[1,\ell]$,
\emph{negative} $d$-moves for the alphabet $[\overline 1,\overline \ell]$, or \emph{mixed} $2$-moves
that interchange two consecutive indices of opposite signs.
Let us say that a double reduced word
of length $d$ is \emph{minimal} if we can perform a generalized $d$-move
on it. {To any two minimal double
words $\mathbf i$ and $\mathbf{i'}$ related by a generalized $d$-move}
$\delta:\mathbf i\mapsto\mathbf{i'}$, we associate a cluster
transformation ${\mu}_{\mathbf{i}\rightarrow\mathbf{i'}}:
{\mathcal X}_{\mathbf i}\rightarrow{\mathcal X}_{\mathbf {i'}}$
in the following way:

{\footnotesize

$${\mu}_{\mathbf{i}\rightarrow\mathbf{i'}}=
\begin{cases}
 \mu_{\binom i{1}} &\text{
if $\delta$ is a move $i\ \overline i\leftrightarrow\overline i\ i$ or a 3-move;}\\
\mu_{\binom i{1}}\mu_{\binom j{1}}\mu_{\binom i{1}}& \text{
if $\delta$ is a 4-move;}\\
\mu_{\binom j{2}}\mu_{\binom i{1}}\mu_{\binom j{1}}
\mu_{\binom j{2}}\mu_{\binom i{2}}\mu_{\binom j{2}}\mu_{\binom i{1}}\mu_{\binom i{2}}
\mu_{\binom j{1}}\mu_{\binom j{2}}&  \text{if $\delta$ is a 6-move;}\\
 \text{the identity map}&\text{otherwise.}
\end{cases}.
$$

}

Since mutations commute with amalgamation, we may extend these definitions
to any two double words $\mathbf i,\mathbf{i'}\in R(u,v)$ related by a generalized $d$-move.
Finally, if  $\mathbf i,\mathbf j$ are {double words} linked by a sequence
of generalized $d$-moves and  $\mathbf{i}\to\mathbf{i_1}\rightarrow
\dots\rightarrow\mathbf{i_{n-1}}\rightarrow\mathbf{j}$ is the associated chain of elements,
we define the cluster transformation
${\mu}_{\mathbf{i}\rightarrow\mathbf{j}}$ as the composition
${\mu}_{\mathbf{i_{n-1}}\rightarrow\mathbf j}\circ \dots\circ
{\mu}_{\mathbf{i}\rightarrow\mathbf{i_1}}$. {A cluster $\mathcal X$-variety
is associated to every double Bruhat cell via the following theorem.}

\begin{thm}[\cite{FGcluster}]\label{fg}For any $u,v\in W$
and $\mathbf i,\mathbf j\in R(u,v)$, we have
$\ev_{\mathbf{i}}=\ev_{\mathbf{j}}\circ{\mu}_{\mathbf{i}\rightarrow\mathbf j}$.
\end{thm}


\subsection{{Cluster $\mathcal X$-varieties} related to $(G, \pi_*)$}
\label{section:cluster*}
We sum-up here some definitions and results of \cite{RB}. To state an analog
of Theorem \ref{thm:ev} and Theorem \ref{fg} for~$(G, \pi_*)$, we need new
evaluation maps and an extension of the combinatorics.
Let us first
recall that the fundamental weights $\omega_i\in{\mathfrak h}^*$
are permuted by the transformation $(-w_0)$.
We denote $i\mapsto i^{\star}$ the induced permutation of the indices of these
weights, that is $\omega_{i^{\star}}=-w_0(\omega_i)$.
{Let $w\in W$ and $s_{i_1}\dots s_{i_n}$
be a reduced decomposition of $w$, then $w^{\star}\in W$ is the element given by
$w^{\star}=s_{i_1^{\star}}\dots s_{i_n^{\star}}$.} (The Tits theorem implies that
this definition doesn't depend on the choice of the reduced expression for $w$.)

For every $w, w'\in W$, we denote $w\rightarrow w'$ if and only if we can find
$i\in[1,\ell]$ such that $w=s_iw'$ and $\ell(w)=\ell(w')+1$, and denote
$\leq$ the \emph{right weak order} on $W$, i.e. $w'\leq w$ if there exists a chain
$w\rightarrow \dots\rightarrow w'. $

\begin{definition}
Let $w_1\leq v, w_2\in W$.
A \emph{$(w_1,w_2)_v$-word} $\mathbf i$
is a double word linked to a product $\mathbf{i_1}\mathbf{i_2}$, with
$\mathbf {i_1}\in R({w_1^{\star}}^{-1},vw_1^{-1})$ and
$\mathbf{i_2}\in R({w_2}^{-1},w_0w_2^{-1})$, by a sequence of mixed $2$-moves.
The product $\mathbf{i_1}\mathbf{i_2}$ is called a \emph{trivial $(w_1,w_2)_v$-word}.
In particular, the set $R({v^{\star}}^{-1},w_0)$ is the set of $(v,e)_v$-words. Let $W(w_1,w_2)_v$
denote the set of $(w_1,w_2)_v$-words and $D(v)$ be the union over $w_1\leq v,w_2\in W$
of all the sets $W(w_1,w_2)_v$. In particular, we have $R({v^{\star}}^{-1},w_0)\subset D(v)$.
\end{definition}

The set $ D(v)$ are going to play in Theorem \ref{thm:ev*1} and Theorem \ref{thm:ev*}
the same role that were playing the sets $R(u,v)$ in Theorem \ref{thm:ev} and
Theorem \ref{fg}.

Now, for every $x\in B_-B$, let $x=[x]_-[x]_0[x]_+=[x]_{\leq 0}[x]_+$ be
its {Gauss decomposition}, that is: $[x]_{\pm}$ belongs to the unipotent parts
of the respective Borel subgroups and $[x]_0$ to the Cartan part of $G$.
Let us denote $\theta: G \rightarrow G$ the Cartan automorphism
given by
$$\begin{array}{cccc}
a^{\theta}=a^{-1}\in H,& {E^i}^{\theta}=F^i,&
{F^i}^{\theta}=E^i\ .
\end{array}$$

\begin{definition}
Let $w_1\leq v,w_2\in W$, $\mathbf i=\mathbf{i_1}\mathbf{i_2}$ be a trivial
$(w_1,w_2)_v$-word and $\mathbf x\in{\mathcal X}_{\mathbf i}$,
$\mathbf{x_1}\in{\mathcal X}_{\mathbf{i_1}}$, $\mathbf {x_2}
\in{\mathcal X}_{\mathbf{i_2}}$ be such that
$\mathbf x=\mathfrak{m}(\mathbf{x_1},\mathbf{x_2})$.
We define the maps $\ev_{\mathbf i}^{\red}:{\mathcal X}_{\mathbf{i_2}}\rightarrow G$
and $\ev^{\mathfrak L}_{\mathbf i},\ev^{\mathfrak R}_{\mathbf i}:{\mathcal X}_{\mathbf i}
\rightarrow G$ by:
$$
\begin{array}{lllll}
\ev_{\mathbf i}^{\red}(\mathbf{x})=\ev_{\mathbf i}(\mathbf x)\displaystyle
\prod_{j\in[1,\ell]}H^j(x_{\binom{j}{N^j(\mathbf i)}}^{-1})\ ,
\end{array}
$$
$$\begin{array}{ccc}
\ev^{\mathfrak R}_{\mathbf i}(\mathbf x)=\ev_{\mathbf{i_1}}(\mathbf{x_1})
[\ev_{\mathbf{i_2}}^{\red}(\mathbf{x_2})\widehat{w_2w_0}]_{\leq 0}&\mbox{and}
&\ev^{\mathfrak L}_{\mathbf i}(\mathbf x)=\ev_{\mathbf{i_1}}(\mathbf{x_1})
[(\ev^{\mathfrak R}_{\mathbf{i_2}}(\mathbf{x_2}))^{\theta}\widehat{w_0}]_{\leq 0}^{\theta}\ .
\end{array}
$$
These maps are then extended to every $\mathbf i\in D(v)$ by setting
$$\begin{array}{ccc}
\ev^{\mathfrak L}_{\mathbf i}=\ev^{\mathfrak L}_{\mathbf{i_1}\mathbf{i_2}}
\circ\mu_{\mathbf i\to\mathbf{i_1}\mathbf{i_2}}&\mbox{and}&
\ev^{\mathfrak R}_{\mathbf i}=\ev^{\mathfrak R}_{\mathbf{i_1}\mathbf{i_2}}
\circ\mu_{\mathbf i\to\mathbf{i_1}\mathbf{i_2}}\ .
\end{array}$$
For every $v\in W$ and $\mathbf i\in D(v)$, we define the \emph{twisted evaluation}
\begin{equation}\label{equ:evhat}
\widehat{\ev}_{\mathbf i}:{\mathcal X}_{[\mathbf i]_{\mathfrak R}}
\rightarrow(G,\pi_*):\mathbf{x}\mapsto
\ev_{\mathbf{i}}^\mathfrak{L}(\mathbf{x})\ev_{\mathbf 1}({\mathbf x(\mathfrak{R})})\widehat{w_0}
{\ev_{\mathbf{i}}^\mathfrak{R}(\mathbf{x})}^{-1}\ ,
\end{equation}
where, for every double word $\mathbf i$, the seed $\mathcal X$-torus ${\mathcal X}_{[\mathbf i]_{\mathfrak R}}$
is the Poisson variety canonically associated to the seed $[\mathbf i]_{\mathfrak R}
=(I(\mathbf i),I_0(\mathbf i),\eta({\mathbf i}),d({\mathbf i}))$ defined by the values
\begin{equation}\label{equ:defJ}\eta(\mathbf i)_{ij}=\left\{
\begin{array}{ll}
\varepsilon(\mathbf i)_{ij}&\text{if } i,j\in I(\mathbf i)\backslash I^{\mathfrak R}_0(\mathbf i)\ ;\\
0& \text{otherwise.}
\end{array}
\right.
\end{equation}
\end{definition}

Finally, for every $t\in H$, let ${\mathcal X}_{[\mathbf{i}]_{\mathfrak{R}}}(t)$ be the subset
of ${\mathcal X}_{[\mathbf i]_{\mathfrak R}}$ obtained by fixing the cluster variables
$\mathbf x({\mathfrak{R}})=\{(x_j)\mid{j\in I^{\mathfrak{R}}_0(\mathbf i)}\}$
via the equality $\ev_{\mathbf 1}(\mathbf x({\mathfrak{R}}))=t$.
Equation (\ref{equ:defJ})
implies that ${\mathcal X}_{\mathbf{i}}(t)$ is a Poisson
subvariety of ${\mathcal X}_{[\mathbf i]_{\mathfrak R}}$.

\begin{thm}\cite[Theorem 7.9]{RB}\label{thm:ev*1} For every $v\in W$, $t\in H$ and $\mathbf i\in D(v)$,
the map $\widehat{\ev}_{\mathbf i}:{\mathcal X}_{\mathbf i}(t)\rightarrow (F_{t,w_0v^{-1}},\pi_*)$
is a Poisson birational isomorphism onto a Zariski open set of~$F_{t,w_0v^{-1}}$.
\end{thm}

\begin{rem}The origin of the formula (\ref{equ:evhat}) comes from morphisms
defined by Evens and Lu in \cite{EL}, generalized in \cite{RB}. (The proof of
Theorem \ref{thm:ev*1} is based on \cite[Corollary 5.11]{EL}.)
\end{rem}

\begin{thm}\cite[Theorem 8.12]{RB}\label{thm:ev*} Let $v\in W$. For every $\mathbf i,\mathbf j\in D(v)$,
there exists a birational Poisson isomorphism ${\widehat{\mu}}_{\mathbf{i}\rightarrow\mathbf j}:
{\mathcal X}_{[\mathbf{i}]_{\mathfrak{R}}}\to{\mathcal X}_{[\mathbf{j}]_{\mathfrak{R}}}$
such that $\widehat{\ev}_{\mathbf{i}}=\widehat{\ev}_{\mathbf{j}}
\circ{\widehat{\mu}}_{\mathbf{i}\rightarrow\mathbf j}$.
\end{thm}

Describing explicitly the birational Poisson isomorphism ${\widehat{\mu}}_{\mathbf{i}
\rightarrow\mathbf j}:{\mathcal X}_{[\mathbf{i}]_{\mathfrak{R}}}
\to{\mathcal X}_{[\mathbf{j}]_{\mathfrak{R}}}$ asks us to extend the given combinatorics
both on double words and on seed-$\mathcal X$-tori. Here is the construction.

\subsubsection{Combinatorics on double words}
We introduce new moves on double words in order to act transitively in $D(v)$.
From now on, let
us regard the map $i\mapsto\overline{i}$ as an involution on $[1,\ell]\cup[\overline{1},\overline{\ell}]$.

\begin{definition}
For every double word
$\mathbf i=i_1\dots i_n$, let ${\mathfrak L}_{i_1}(\mathbf i)$ (resp. ${\mathfrak R}_{i_n}(\mathbf i)$)
be the double word obtained by changing the first letter
(resp. last letter) of $\mathbf i$ such that:
$$\begin{array}{ccc}
{\mathfrak L}_{i_1}(\mathbf i)=\overline{i_1}i_2\dots i_n
&\mbox{and}&{\mathfrak R}_{i_n}(\mathbf i)=i_1\dots i_{n-1}\overline{i_n}\ .
\end{array}
$$
The involutive map $\mathbf i\mapsto {\mathfrak L}_{i_1}(\mathbf i)$ (resp.
$\mathbf i\mapsto {\mathfrak R}_{i_n}(\mathbf i)$) is called a
\emph{left $\tau$-move} (resp. \emph{right $\tau$-move}).
\end{definition}

Let us fix $w_1\leq v\in W$ and denote $D_{w_1}(v)$ the set of $(w_1,w_2)_v$-words
for every $w_2\in W$. Therefore, the set $D(v)$ is the union
over all $w_1\leq v\in W$ of the sets $D_{w_1}(v)$.

\begin{lemma}\cite{RB}\label{lemma:relPWR}
For every $\mathbf i,\mathbf j\in D_{w_1}(v)$, there exist a sequence
$\varphi_{\mathbf i\to\mathbf j}$ of generalized $d$-moves and right
$\tau$-moves such that $\mathbf j=\varphi_{\mathbf i\to\mathbf j}(\mathbf i)$.
\end{lemma}

\begin{definition}
Let $\mathbf i=i_1\dots i_m\in R(1,w_0)\cup R(w_0,1)$ be a positive or
negative reduced word associated to $w_0$ and $\mathbf j=j_1\dots j_n$ be a
double word. The following \emph{dual-move} $\Delta_{j}$ associated to the
last letter of $\mathbf j$ transforms the product $\mathbf j\mathbf i$ into
the double word
$$\Delta_{j}:\mathbf j\mathbf i\mapsto j_1\dots j_{n-1}\overline{j_n}^{\star}\
\overline{i_m}\dots \overline{i_1}\ ,$$
where $j=j_n$ if $j_n$ positive and $j=\overline{j_n}^{\star}$
if $j_n$ negative. It is easy to see that $\Delta_{j}\circ \Delta_{j}$ is the identity map.
\end{definition}

\begin{definition}
Let $v\in W$ and $\mathbf i\in D(v)$ be a double word. A \emph{$\widehat{d}$-move}
on $\mathbf i$ is a generalized $d$-move; or a right ${\tau}$-move; or
a dual-move $\Delta_{i}$.
\end{definition}

Here is an analog of the Tits theorem for the set $D(v)$.

\begin{lemma}\cite{RB} If $\mathbf i,\mathbf j\in D(v)$ then there exists a sequence
of $\widehat{d}$-moves relating $\mathbf i$ and $\mathbf j$.
\end{lemma}

\subsubsection{Tropical mutations}
\label{section:tropical}
The $\tau$-moves
lead to a new type of mutations on seed $\mathcal X$-tori, called
{tropical mutations} and defined in the following way.
Let $b_{ij}$ be the numerator of $\varepsilon_{ij}$ for every $i,j\in I$;
so we have $b_{ij}=\varepsilon_{ij}$ unless $i,j\in I_0$. Let us
suppose that the denominator of $\varepsilon_{ij}$ is the same for every
$i,j\in I_0$. (Using the formulas (\ref{equ:exeps}) and (\ref{equ:amaleps}),
it is clear that it is the case for any seed $\mathbf I(\mathbf i)$ associated
to a double word $\mathbf i$.)

\begin{definition}Let $\mathbf I=(I_0,I,\varepsilon,d)$ be a seed
such that $I_0$ is not empty.
A \emph{cover} $\mathfrak{C}$ on $\mathbf I$ is a family of sets
$I_1,\dots,I_n\subset I_0$ such that $I_0=\cup_{i=1}^{n}I_i$.
(The union is not necessary disjoint.)
For every $k\in I_0$, we denote $I_0(k)$ the union
$$I_0(k):=\bigcup_{\{i\mid k\in I_i\}}I_i\ .$$
\end{definition}

\begin{definition}\label{def:trop}
Let $\mathbf{I}=(I,I_0,\varepsilon,d)$ and $\mathbf {I'}=(I',{I'}_0,{\varepsilon}',d')$
be two seeds with covers, and $k\in I_0$. A $\emph{tropical mutation}$ $\emph{in the direction k}$
is an involution $\mu_k:\mathbf I\longrightarrow\mathbf{I'}$
satisfying the following conditions:

(i)$\mu_k(I_0(i))={I'}_0(i)$;

(ii) ${d'}_{\mu_k(i)}=d_i$;

{(iii)$$
\varepsilon'_{\mu_k(i)\mu_k(j)} =
\begin{cases}
-\varepsilon_{ij} & \text{if $i=k$ or $j=k$;} \\[.05in]
\varepsilon_{ij}
& \text{if $i,j\in I_0(k)\backslash\{k\}$};\\
\varepsilon_{ij} -\varepsilon_{ik}b_{kj} &\text{otherwise.}
\end{cases}
$$
Tropical mutations induce involutive maps between the corresponding seed $\mathcal X$-tori,
which are denoted by the same symbols $\mu_k$ and given by
\begin{equation}\label{equ:formuletropmut}
x_{\mu_k(i)}=\left\{ \begin{array}{lll}
{x_k}^{-1}& \mbox{if}\ i=k;\\
x_ix_k^{b_{ki}}& \mbox{if}\ i\in I_0(k)\backslash\{k\};\\
x_i& \mbox{otherwise}.
\end{array}\right.
\end{equation}}
For the remaining part of the paper, we suppose mutations and symmetries respect covers on seeds.
A \emph{generalized cluster transformation} linking two seeds
(and two seed $\mathcal X$-tori) is then a composition of symmetries, mutations, and
tropical mutations.
\end{definition}

Let $\mathbf i$ be a double word and recall the subsets
$I_0^{\mathfrak{L}}(\mathbf i)$ and $I_0^{\mathfrak{R}}(\mathbf i)$ of $I_0(\mathbf i)$
defined in Subsection \ref{section:clusterpiG}. From now on, the seed $\mathbf{I(i)}=
(I(\mathbf i),I_0(\mathbf i),\varepsilon(\mathbf i),d(\mathbf i))$ is given with the cover
$$I_0(\mathbf i)=I_0^{\mathfrak{L}}(\mathbf i)\cup I_0^{\mathfrak{R}}(\mathbf i)\ .$$

\begin{lemma}\cite[Proposition 5.11]{RB}\label{prop:muttrop} The following tropical mutations
are Poisson birational isomorphisms for every double word $\mathbf i=i_1\dots i_n$.
They are called respectively \emph{left} and \emph{right tropical mutations}.
$$\begin{array}{lcr}
\mu_{\binom {i_1}{0}}:{\mathcal X}_{\mathbf i}\rightarrow {\mathcal X}_{{\mathfrak L}_{i_1}(\mathbf i)}
&\mbox{and}&
\mu_{\binom {i_n}{N^{i_n}(\mathbf i)}}:{\mathcal X}_{\mathbf i}\rightarrow {\mathcal X}_{{\mathfrak R}_{i_n}(\mathbf i)}\ .
\end{array}
$$
\end{lemma}

Left and right tropical mutations are related to the geometry of the group $G$ in the following way.

\begin{lemma}\cite[Proposition 5.24]{RB}\label{equ:trop} The following equalities
are satisfied for every $u,v\in W$, every $(u,v)$-adapted double word
$\mathbf i=i_1\dots i_n\in R(u,v)$, and every $\mathbf x\in{\mathcal X}_{\mathbf i}$.
$$\begin{array}{ccc}
[\ev_{\mathbf i}(\mathbf x)\widehat{v^{-1}}]_{\leq 0}&=&
[\ev_{{\mathfrak R}_{i_n}(\mathbf i)}\circ\mu_{\binom {i_n}{N^{i_n}(\mathbf i)}}
(\mathbf x)\widehat{s_{i_n}v^{-1}}]_{\leq 0}\\
{[\widehat{u}^{-1}\ev_{\mathbf i}(\mathbf x)]_{\geq 0}}&=&
[\widehat{s_{i_1}u}^{-1}\ev_{{\mathfrak L}_{i_1}(\mathbf i)}\circ\mu_{\binom {i_1}{0}}
(\mathbf x)]_{\geq 0}\ .
\end{array}$$
\end{lemma}

This result is the first step to get the combinatorics of the maps
$b\mapsto[b\widehat{v^{-1}}]_{\leq 0}$ and $c\mapsto[\widehat{u}^{-1}c]_{\geq 0}$},
with $b\in G^{1,v}$ and $c\in G^{u,1}$. To do that, we need more notations.
For every positive reduced word $\mathbf i=i_1\dots i_n$, every  negative reduced word
$\mathbf j=\overline{j_1}\dots \overline{j_n}$, and every $k\in[1,n+1]$, we then introduce
the double words
\begin{equation}\label{equ:xiword}
\begin{array}{ccllll}
\mathbf i(k)=\mathbf i(k)_-\mathbf i(k)_+& \mbox{where}&\mathbf i(k)_+
=i_1\dots i_{k-1} &\mbox{and}& \mathbf i(k)_-=\overline{i_n}\dots\overline{i_{k}}\\
\mathbf j(k)=\mathbf j(k)_-\mathbf j(k)_+& \mbox{where}&\mathbf j(k)_+
=j_{k-1}\dots j_{1} &\mbox{and}& \mathbf j(k)_-=\overline{j_{k}}\dots\overline{j_{n}}
\end{array}.
\end{equation}
In particular, we will use the notations
$$\begin{array}{ccc}
\mathbf{i^{\square}}:=\mathbf i(1)=
\overline{i_n}\dots\overline{i_{1}}& \mbox{and}& \mathbf{j^{\square}}:=\mathbf j(n+1)=
{j_n}\dots{j_{1}}\ .
\end{array}$$
For every positive reduced word
$\mathbf i=i_1\dots i_n$, every  negative reduced word
$\mathbf j=\overline{j_1}\dots \overline{j_n}$, and every $k\in[1,n]$, we define the
generalized cluster transformations $\zeta_{\mathbf i(k)}:{\mathcal X}_{\mathbf i(k)}
\to{\mathcal X}_{\mathbf i(j-1)}$, $\zeta_{\mathbf i}:
{\mathcal X}_{\mathbf{i}}\to{\mathcal X}_{\mathbf{i^{\square}}}$, and $\zeta_{\mathbf j(k)}:{\mathcal X}_{\mathbf j(k)}
\to{\mathcal X}_{\mathbf j(k-1)}$, $\zeta_{\mathbf j}:
{\mathcal X}_{\mathbf{j}}\to{\mathcal X}_{\mathbf{j^{\square}}}$  by the following formulas
\begin{equation}\label{equpartialzeta}
\begin{array}{lcc}
\zeta_{\mathbf i(k)}=\mu_{\binom{i_k}{N^{i_k}(\mathbf i(k)_-)}}\circ\dots
\circ
\mu_{\binom{i_k}{N^{i_k}(\mathbf i)-1}}\circ\mu_{\binom{i_k}{N^{i_k}(\mathbf i)}}&
\mbox{and}&\zeta_{\mathbf i}=\zeta_{\mathbf i(1)}\circ\dots\circ\zeta_{\mathbf i(n)}\\
\\
\zeta_{\mathbf j(k)}=\mu_{\binom{j_k}{N^{j_k}(\mathbf j(k)_-)-1}}\circ\dots
\circ\mu_{\binom{j_k}{1}}\circ
\mu_{\binom{j_k}{0}}&
\mbox{and}&\zeta_{\mathbf j}=\zeta_{\mathbf j(n)}\circ\dots\circ\zeta_{\mathbf j(1)}
\end{array}\ .
\end{equation}

The following result is easily deduced from Theorem \ref{fg} and Lemma \ref{equ:trop}.

\begin{cor}\label{lemma:twist} Let $u,v\in W$, $\mathbf i\in R(1,v)$ and
$\mathbf j\in R(u,1)$. For every $\mathbf x\in{\mathcal X}_{\mathbf{i}},
\mathbf y\in{\mathcal X}_{\mathbf{j}}$ the following equalities are satisfied:
\begin{equation}\label{equ:torev}
\begin{array}{ccc}
[\ev_{\mathbf i}(\mathbf x)\widehat{v^{-1}}]_{\leq 0}=\ev_{\mathbf{i^{\square}}
}\circ\zeta_{\mathbf i}(\mathbf x)&\mbox{and}&
{[\widehat{u}^{-1}\ev_{\mathbf j}(\mathbf y)]_{\geq 0}}
=\ev_{\mathbf{j^{\square}}}\circ\zeta_{\mathbf j}(\mathbf y)\ .
\end{array}
\end{equation}
\end{cor}

\subsubsection{The extended combinatorics on seed $\mathcal X$-tori}
Let us first remark that equation (\ref{equ:defJ}) implies that any cluster
transformation
$\mu_{\mathbf i\to\mathbf j}:{\mathcal X}_{\mathbf i}\to{\mathcal X}_{\mathbf j}$,
for any $\mathbf j$ linked to $\mathbf i$ by composition of generalized $d$-moves,
induces a cluster transformation
$\mu_{[\mathbf i]_{\mathfrak{R}}\to[\mathbf j]_{\mathfrak{R}}}
:{\mathcal X}_{[\mathbf i]_{\mathfrak{R}}}\to{\mathcal X}_{[\mathbf j]_{\mathfrak{R}}}$ given by
\begin{equation}\label{equ:mucrochet}
x_{\mu_{[\mathbf i]_{\mathfrak{R}}\to[\mathbf j]_{\mathfrak{R}}}(j)}=\left\{
\begin{array}{ll}
x_{\mu_{\mathbf i\to\mathbf j}(j)}&\text{if }j\in I(\mathbf i)\backslash I^{\mathfrak R}_0(\mathbf i)\ ;\\
x_j&\text{if }j\in I^{\mathfrak R}_0(\mathbf i)\ .
\end{array}
\right.
\end{equation}
We have in particular the restriction on Poisson subvariety $\mu_{[\mathbf i]_{\mathfrak{R}}
\to[\mathbf j]_{\mathfrak{R}}}:{\mathcal X}_{\mathbf{i}}(t)\to{\mathcal X}_{\mathbf{j}}(t)$
for any $t\in H$.
It is easy to see that for every double word the identity map is a Poisson morphism from
${\mathcal X}_{[\mathbf i]_{\mathfrak{R}}}$ to ${\mathcal X}_{[\mathfrak{R}
(\mathbf i)]_{\mathfrak{R}}}$. We denote it $\mu_{[\mathbf i]_{\mathfrak{R}}\to
[\mathfrak{R}(\mathbf i)]_{\mathfrak{R}}}$ and extend the definition (\ref{equ:mucrochet})
by the formula
$$\mu_{[\mathbf i]_{\mathfrak{R}}\to[\mathfrak{R}(\mathbf j)]_{\mathfrak{R}}}:=
\mu_{[\mathbf i]_{\mathfrak{R}}\to[\mathbf j]_{\mathfrak{R}}}\circ
\mu_{[\mathbf j]_{\mathfrak{R}}\to[\mathfrak{R}(\mathbf j)]_{\mathfrak{R}}}\ .$$
And because
every elements $\mathbf i,\mathbf j\in D_{w_1}(v)$ can be obtained from each
other by a sequence of generalized $d$-moves and right $\tau$-moves by Lemma
\ref{lemma:relPWR}, we deduce from Theorem~\ref{fg} and the expression
(\ref{equ:evhat}) of twisted evaluations that for every $\mathbf i,\mathbf j
\in D_{w_1}(v)$, we have the equality
\begin{equation}\label{equ:taumuR}
\widehat{\ev}_{\mathbf{i}}=\widehat{\ev}_{\mathbf{j}}
\circ{{\mu}}_{[\mathbf{i}]_{\mathfrak{R}}\rightarrow[\mathbf j]_{\mathfrak{R}}}\ .
\end{equation}
Therefore, the equation (\ref{equ:taumuR}) allows us to associate, for every $w\leq v\in W$,
a cluster $\mathcal X$-variety ${\mathcal X}_{w\leq v}$ to the set $D_w(v)$: it
indeed contains every seed $\mathcal X$-torus ${\mathcal X}_{[\mathbf i]_{\mathfrak{R}}}$ when
$\mathbf i\in D_w(v)$.

Next, we consider the combinatorics on seed $\mathcal X$-tori related to dual moves
to link all the cluster $\mathcal X$-varieties ${\mathcal X}_{w\leq v}$ for a fixed
$v\in W$. To do that, we define for every double reduced word $\mathbf j$, every positive word
$\mathbf i_+\in R(1,w_0)$ and every $i\in[1,\ell]$, the
map $\Xi_{k}:{\mathcal X}_{[\mathbf j\mathbf{i_+}
\overline{k}]_{\mathfrak R}}\to{\mathcal X}_{[\mathbf j\mathbf i_+^{\square}
k^{\star}]_{\mathfrak R}}$ given by
$$x_{\Xi_{k}\binom{i}{j}}=\left\{
\begin{array}{lll}
x_{\zeta_{\mathbf i_+}\binom{i}{j}}&\mbox{if }j<N^i(\mathbf j\mathbf i_+)\ ;\\
x_{\zeta_{\mathbf i_+}\binom{i}{j}}x_{\binom{i^{\star}}{N^i(\mathbf j\mathbf i_+\overline{k})}}^{-1}
&\mbox{if }j=N^i(\mathbf j\mathbf i_+)<N^i(\mathbf j\mathbf i_+\overline{k})\ ;\\
x_{\binom{i}{j}}&\mbox{otherwise}\ .
\end{array}
\right.
$$
Let $i\in[1,\ell]$ and $\mathbf i$ be a double word such that we can
apply the dual move $\Delta_i$ on it. Then the following product
is a birational Poisson isomorphism.
$$
\begin{array}{cccl}
\Xi_{s_i}:&{\mathcal X}_{[\mathbf i]_{\mathfrak R}}
&\longrightarrow&{\mathcal X}_{[\Delta_i(\mathbf i)]_{\mathfrak R}}\\
&\mathbf x&\longmapsto&\mu_{[\mathbf i_+^{\square}i^{\star}]_{\mathfrak R}\to[\Delta_{i}
(\mathbf i)]_{\mathfrak R}}\circ\Xi_{i}\circ\mu_{[\mathbf i]_{\mathfrak R}\to
[\mathbf i_+\overline{i}]_{\mathfrak R}}(\mathbf x)\ .
\end{array}
$$

\begin{lemma}\cite[Proposition 8.11]{RB}\label{prop:salt} For every
$i\in[\overline{1},\overline{\ell}]$
and every double reduced word $\mathbf i\in R(s_i,w_0)$ starting with
the letter $i$, we have the equality
$\widehat{\ev}_{\mathbf i}=\widehat{\ev}_{\Delta_{i}(\mathbf i)}
\circ\Xi_{s_i}$.
\end{lemma}

This result is then immediately extended to any trivial $(w_1,1)$-words,
for any $w_1\leq v\in W$ by using the amalgamation product on seed $\mathcal X$-tori.

Here is finally the construction of the birational Poisson isomorphism
$\widehat{\mu}_{\mathbf i\rightarrow \mathbf{j}}$ associated to any $\mathbf i,\mathbf j\in D(v)$.
We have seen that to any double words $\mathbf i,\mathbf{i'}\in D(v)$,
there exists a $\widehat{d}$-move $\delta$ such that $\delta:\mathbf{i}
\rightarrow\mathbf{i'}$. From the previous results, it is natural to associate
to $\delta$ the birational Poisson isomorphism $\widehat{\mu}_{\mathbf i\rightarrow \mathbf{i'}}:
{\mathcal X}_{\mathbf i}\rightarrow{\mathcal X}_{\mathbf{i'}}$ given by
\begin{itemize}
\item
the cluster transformation $\mu_{[\mathbf i]_{\mathfrak R}\rightarrow [\mathbf{i'}]_{\mathfrak R}}$
if $\delta$ is a generalized $d$-moves;
\item
the identity map if $\delta$ is a right $\tau$-move;
\item
the map $\Xi_{s_i}$ if $\delta$ is the dual-move $\Delta_{i}$.
\end{itemize}
This definition is then extended to every $\mathbf i,\mathbf{j}\in  D(v)$ in
the usual way: if  $\mathbf i,\mathbf j$ are double words linked be a sequence
of $\widehat{d}$-moves and  $\mathbf{i}\to\mathbf{i_1}\rightarrow
\dots\rightarrow\mathbf{i_{n-1}}\rightarrow\mathbf{j}$ is the associated chain of elements,
the map
$\widehat{{\mu}}_{\mathbf{i}\rightarrow\mathbf{j}}$ will be the composition
$\widehat{{\mu}}_{\mathbf{i_{n-1}}\rightarrow\mathbf j}\circ \dots\circ
\widehat{{\mu}}_{\mathbf{i}\rightarrow\mathbf{i_1}}$, and Theorem~\ref{thm:ev*} is derived
from the equality (\ref{equ:taumuR}) and Lemma \ref{prop:salt}.

We finish this subsection by considering the dual Poisson Lie-group
$(BB_-,\pi_*)$ in $(G,\pi_*)$. For every double word $\mathbf i$, let
${\mathcal X}_{\mathbf i}^{\dual}\subset{\mathcal X}_{[\mathbf i]_{\mathfrak R}}$
be such that $x_i\neq x_j$ for every $i,j\in I^{\mathfrak{R}}_0(\mathbf i)$.
It is a Poisson submanifold of ${\mathcal X}_{[\mathbf i]_{\mathfrak R}}$
because of (\ref{equ:defJ}). The decomposition (\ref{equ:decG^*})
then leads to the following result.

\begin{cor}For every $\mathbf i\in D(w_0)$,
the map $\widehat{\ev}_{\mathbf i}:{\mathcal X}_{\mathbf i}^{\dual}
\rightarrow (BB_-,\pi_*)$ is a Poisson birational isomorphism on a Zariski open
set of $BB_-$ and the equality $\widehat{\ev}_{\mathbf{i}}=\widehat{\ev}_{\mathbf{j}}
\circ{\widehat{\mu}}_{\mathbf{i}\rightarrow\mathbf j}$ is satisfied for every $\mathbf j\in D(w_0)$.
\end{cor}

\section{The combinatorics of De-Concini-Kac-Procesi automorphisms}
\label{section:Artin}
We add the left tropical mutations of Lemma \ref{prop:muttrop} to
the birational Poisson isomorphisms of the previous section in
order to define an action
of the Artin group associated to $\mathfrak{g}$ on seed $\mathcal X$-tori.
Applying twisted evaluations then leads to a generalization of the
De-Concini-Kac-Procesi automorphisms associated to parabolic subgroups of
$W$.

\subsection{Artin groups actions and Poisson automorphisms on seed $\mathcal X$-tori}
Let us denote $T_{ij}^{(m)}$ the product of $m$ factors ${T_i}{T_j}{T_i}\dots$ and
let $m_{ij}:=2,3,4,6$ when $a'_{ij}a'_{ji}=0,1,2,3$, respectively,
where $A'=(a'_{ij})$ denotes the Cartan matrix associated to any semi-simple Lie algebra
$\mathfrak{g}'$. We recall that the \emph{Artin group} ${\mathcal B}_{\mathfrak g'}$
{associated to} $\mathfrak{g}'$ is given by the presentation
$${\mathcal B}_{\mathfrak g'}=<T_1,\dots,T_{\ell'}\mid T_{ij}^{(m_{ij})}=T_{ji}^{(m_{ji})}>\ .$$
We are going to show that tropical mutations induced by left $\tau$-moves, mutations
and the previous saltations lead to Poisson automorphisms that define an action of the
Artin group associated to $\mathfrak{g}$ on cluster $\mathcal X$-varieties.

Let us remember the involution $\star:i\mapsto i^{\star}$ on the indices of
fundamental weights induced by the transformation $(-w_0)$ on ${\mathfrak h}^*$ given in Subsection \ref{section:cluster*}.
Let $W_I$ be the parabolic subgroup of $W$ generated by a subset $I\subset[1,\ell]$ stable
under the involution $\star$, ${\mathfrak{g}}^I$ be the associated semi-simple Lie
algebra, and $w_0(I)$ be the longest element of $W_I$. In particular, for every $j\in I$,
there exist a reduced expression of $w_0(I)$ starting with $j$. Let us then remark that, to
every $\mathbf i\in D(w_0(I))$ and every letter $j\in I$, we can associate via Theorem
\ref{thm:ev*} a double reduced word $\mathbf{i_0}\in R(w_0(I),w_0)$
starting with $\overline{j}$ such that the equality $\widehat{\ev}_{\mathbf i}=
\widehat{\ev}_{\mathbf{i_0}}\circ\widehat{\mu}_{\mathbf i\rightarrow\mathbf{i_0}}$ is satisfied.
We use this double reduced word $\mathbf{i_0}$ (although the final result doesn't depend of this
particular choice) to define the Poisson birational automorphism
\begin{equation}\label{equ:defti}
\begin{array}{ccc}
{\mathcal T}_j({\mathbf i}):{\mathcal X}_{[\mathbf{i}]_{\mathfrak R}}
\longrightarrow{\mathcal X}_{[\mathbf{i}]_{\mathfrak R}}:&\mathbf x
\longmapsto\widehat{\mu}_{{\mathfrak L}(\mathbf {i_0})\rightarrow\mathbf i}
\circ{\mu_{\binom j{0}}}\circ\widehat{\mu}_{\mathbf {i}\rightarrow\mathbf{i_0}}
\end{array}\ .
\end{equation}

Here is how we relate braids to cluster combinatorics.

\begin{thm}\label{thm:clustertresses} For every $I\subset[1,\ell]$ stable
under the involution $\star$ and $\mathbf i\in D(w_0(I))$, the maps
${\mathcal T}_j({\mathbf i}):{\mathcal X}_{[\mathbf i]_{\mathfrak R}}\longrightarrow
{\mathcal X}_{[\mathbf i]_{\mathfrak R}}$, $j\in I$, define an action of
${\mathcal B}_{{\mathfrak g}^I}$ on ${\mathcal X}_{[\mathbf i]_{\mathfrak R}}$
given by
$$\begin{array}{rlccc}
{\mathcal B}_{{\mathfrak g}^I}\times{\mathcal X}_{[\mathbf i]_{\mathfrak R}}
\longrightarrow{\mathcal X}_{[\mathbf i]_{\mathfrak R}}:&
(T_j,\mathbf x)\longmapsto{\mathcal T}_j({\mathbf i})(\mathbf x)
\end{array}\ .$$
\end{thm}

We need two lemmas to prove this theorem. To any set $I\subset[1,\ell]$ stable under the involution $\star$,
$w\leq w_0(I)\in W$, and any reduced decomposition $i_1\dots i_n$ of $w$,
and any double word $\mathbf i\in D(w_0(I))$, we associate a Poisson automorphism
$\mathcal{T}_{w}(\mathbf i)$ on ${\mathcal X}_{[\mathbf{i}]_{\mathfrak R}}$
given by
$$\mathcal{T}_{w}(\mathbf i)=\mathcal{T}_{i_n}(\mathbf i)\circ\dots
\circ\mathcal{T}_{i_1}(\mathbf i)\ .$$
For every $v\in W$, $k\in[1,\ell(v)]$ and $\mathbf{i_0}=\mathbf{i_1}\mathbf{i_2}\in D(v)$
such that $\mathbf{i_1}\in R(1,v)$ and $\mathbf{i_2}\in R(1,w_0)$, we recycle the
notation of Subsection \ref{section:tropical} by denoting
$\mathbf {i_0}(k):=\mathbf {i_1}(k)\mathbf{i_2}$.
Let us also introduce a variation of (\ref{equpartialzeta}) by setting:
$$
\begin{array}{ccc}
\zeta_{\mathbf{i_0}(j)}:=\mu_{\binom{i_j}{N^{i_j}(\mathbf{i_1}(j)_--1)}}\circ\dots
\circ\mu_{\binom{i_j}{1}}\circ\mu_{\binom{i_j}{0}}&
\mbox{and}&\zeta_{\mathbf{i_0}(\leq k)}:=\zeta_{\mathbf{i_0}(k)}\circ\dots\circ\zeta_{\mathbf{i_0}(1)}\ .
\end{array}
$$

\begin{lemma}\label{lemma:Tijtwist}For every $I\subset[1,\ell]$ stable
under the involution $\star$, every $w\leq w_0(I)$, every $\mathbf i\in D(w_0(I))$
and every $(e,e)$-word $\mathbf{i_0}\in D(v)$ as above, we have the following equality
$$
\mathcal{T}_{w}(\mathbf i)=\widehat{\mu}_{\mathbf{i_0}(\ell(w))\to\mathbf i}
\circ\zeta_{\mathbf{i_0}(\leq\ell(w))}\circ\widehat{\mu}_{\mathbf i\to\mathbf{i_0}}\ .
$$
\end{lemma}
\begin{proof}The proof of this lemma is done by induction on the length of
$w\in W$, by using the equality $\zeta_{\mathbf i(\leq k)}
=\zeta_{\mathbf i(k)}\circ\zeta_{\mathbf i(\leq k-1)}$. 
The first step of this induction, that is when $\ell(w)=1$, comes from the definition 
(\ref{equ:defti}) of the automorphism $\mathcal{T}_j$.
\end{proof}

\begin{lemma}\label{lemma:tormut} For every $v\in W$ and every reduced words
$\mathbf i,\mathbf j\in R(1,v)$, or $\mathbf i,\mathbf j\in R(v,1)$, we have the equality
$\mu_{\mathbf{i^{\square}}\rightarrow
\mathbf{j^{\square}}}\circ\zeta_{\mathbf i}=\zeta_{\mathbf j}\circ
\mu_{\mathbf i\rightarrow\mathbf j}$.
\end{lemma}
\begin{proof}We suppose that $\mathbf i,\mathbf j\in R(1,v)$.
Let us recall that the involution $\square$ maps double reduced
words to double reduced words, and that the evaluation map $\ev_{\mathbf j}$ associated
to any double reduced word $\mathbf j$ is birational because of Theorem
\ref{thm:evG}. Therefore an equality $\mathbf y=\mathbf z$ between cluster variables on
${\mathcal X}_{\mathbf{i^{\square}}}$ is satisfied if and only if
the equality $\ev_{\mathbf{i^{\square}}}(\mathbf y)
=\ev_{\mathbf{i^{\square}}}(\mathbf z)$ is satisfied on $G$. Now, it suffices
to apply Theorem \ref{fg} and the second equation of (\ref{equ:torev}) to
obtain the following equality for every $\mathbf x\in {\mathcal X}_{\mathbf i}$.
The case $\mathbf i,\mathbf j\in R(v,1)$ is proved in the same way.
$$\ev_{\mathbf{i^{\square}}}\circ\zeta_{\mathbf i}(\mathbf x)
=[\ev_{\mathbf{j}}\circ\mu_{\mathbf i\rightarrow\mathbf j}
(\mathbf x)\widehat{v^{-1}}]_{\leq 0}=\ev_{\mathbf{j^{\square}}}
\circ\zeta_{\mathbf j}\circ\mu_{\mathbf i\to \mathbf j}(\mathbf x)
=\ev_{\mathbf{i^{\square}}}\circ\mu_{\mathbf{j^{\square}}\rightarrow
\mathbf{i^{\square}}}\circ\zeta_{\mathbf j}\circ
\mu_{\mathbf i\rightarrow\mathbf j}(\mathbf x)\ .$$
\end{proof}

We can now prove Theorem \ref{thm:clustertresses}.

\begin{proof}Theorem \ref{thm:clustertresses} is clearly true if the set $I$ contains only one
element, so let us take $i,j\in I$ such that $i\neq j$. Let us recall that
for every $I\subset[1,\ell]$ and every $w\in W_I$, we have the relation $w\leq w_0(I)$.
Therefore, there exist reduced expressions associated to
the element $w_0(I)\in W$ such that the $m_{ij}^{th}$ first letters are the
strings $\mathbf i(ij):={\overline i}\ {\overline j}\ {\overline i}\dots$ and
$\mathbf i(ji):={\overline j}\ {\overline i}\ {\overline j}\dots$ . Now, let
$\mathbf{i_0},\mathbf{j_0}\in R(w_0(I),w_0)$ be two double reduced words such
that their $m_{ij}^{\th}$ first letters are respectively the reduced words
$\mathbf i(ij)$ and $\mathbf i(ji)$, and let us denote $w_{ij},w_{ji}\in W$
the elements associated to $\mathbf i(ij)$ and $\mathbf i(ji)$. To prove the
theorem, we have thus to prove the equality
$\mathcal{T}_{w_{ij}}(\mathbf i)=\mathcal{T}_{w_{ji}}(\mathbf i)$.
To do it, we proceed in several steps, where each one proves the commutativity
of a given diagram. First of all,
applying Lemma~\ref{lemma:tormut} on the reduced words $\mathbf i(ij)$
and $\mathbf i(ji)$, we get the equality
$$\mu_{\mathbf i(ij)^{\square}\rightarrow\mathbf i(ji)^{\square}}
\circ\zeta_{\mathbf i(ij)}=\zeta_{\mathbf i(ji)}\circ
\mu_{\mathbf i(ij)\rightarrow\mathbf i(ji)}\ .$$
This relation is then extended to the double reduced words $\mathbf{i_0}$ and
$\mathbf{j_0}$, because mutations and left tropical mutations commute with
amalgamations done on the right. More precisely, let $\mathbf{i_1}$ and
$\mathbf{j_1}$ be the double reduced words such that $\mathbf{i_0}=\mathbf i(ij)\mathbf{i_1}$
and $\mathbf{j_0}=\mathbf i(ji)\mathbf{j_1}$. We introduce the notation
$\mathbf{i_0^{\boxtimes}}:=\mathbf i(ij)^{\square}\mathbf{i_1}$ and
$\mathbf{j_0^{\boxtimes}}:=\mathbf i(ji)^{\square}\mathbf{j_1}$
and get
$${\zeta_{\mathbf{i_0}(\leq\ell(w_{ij}))}}={\mu}_{\mathbf{j_0^{\boxtimes}}\to{\mathbf{i_0^{\boxtimes}}}}
\circ{\zeta_{\mathbf{j_0}(\leq\ell(w_{ji}))}}\circ{\mu}_{\mathbf{i_0}\to\mathbf {j_0}}\ .$$
Because of (\ref{equ:mucrochet}) and the equalities $\widehat{\mu}_{\mathbf{i_0}\to\mathbf {j_0}}
={\mu}_{[\mathbf{i_0}]_{\mathfrak{R}}\to[\mathbf {j_0}]_{\mathfrak{R}}}$ and 
$\widehat{\mu}_{\mathbf{j_0^{\boxtimes}}\to{\mathbf{i_0^{\boxtimes}}}}=
{\mu}_{[\mathbf{j_0^{\boxtimes}}]_{\mathfrak{R}}\to{[\mathbf{i_0^{\boxtimes}}]_{\mathfrak{R}}}}$, the previous equality implies the commutativity of the diagram
\begin{equation}\label{equ:synth1}\xymatrix{
{{\mathcal X}_{[\mathbf {i_0}]_{\mathfrak R}}}\ar@/_1pc/[d]_{{\zeta_{\mathbf{i_0}(\leq\ell(w_{ij}))}}}
\ar@/^1pc/[r]^{\widehat{\mu}_{\mathbf{i_0}\to\mathbf {j_0}}}
&{\mathcal X}_{[\mathbf {j_0}]_{\mathfrak R}}\ar@/^1pc/[d]^{{\zeta_{\mathbf{j_0}(\leq\ell(w_{ji}))}}}\\
{\mathcal X}_{[\mathbf {i_0^{\boxtimes}}]_{\mathfrak R}}
&{\mathcal X}_{[\mathbf {j_0^{\boxtimes}}]_{\mathfrak R}}\ar@/^1pc/[l]^{\widehat{\mu}_{\mathbf{j_0^{\boxtimes}}\to{\mathbf{i_0^{\boxtimes}}}}}
}
\end{equation}

Next, let $J$ be the maximal set, for the inclusion map,
of pairwise disjoint elements of the set $\{i,j,i^{\star},j^{\star}\}$.
(For example, we have $J=\{i,j\}$ if the involution $\star$ is the identity map,
$J=\{i,j,j^{\star}\}$ if $i=i^{\star}\neq j^{\star}\neq j$, and so on).
Lemma \ref{lemma:Tijtwist}, applied on $J$, then leads to the commutativity of the following diagrams.
\begin{equation}\label{equ:synth2}
\begin{array}{ccccc}\xymatrix{
{\mathcal X}_{[\mathbf i]_{\mathfrak R}}\ar@/_1pc/[d]_{{{\mathcal T}_{w_{ij}}(\mathbf i)}}
\ar@/^1pc/[r]_{\widehat{\mu}_{\mathbf{i}\to\mathbf {i_0}}}
&{{\mathcal X}_{[\mathbf {i_0}]_{\mathfrak R}}}\ar@/_1pc/[d]_{{\zeta_{\mathbf{i_0}(\ell(w_{ij}))}}}
\\
{\mathcal X}_{[\mathbf i]_{\mathfrak R}}\ar@/_1pc/[r]^{{\widehat\mu}_{\mathbf i\to{\mathbf{i_0^{\boxtimes}}}}}
&{\mathcal X}_{[\mathbf {i_0^{\boxtimes}}]_{\mathfrak R}}
}
&&&&\xymatrix{
{\mathcal X}_{[\mathbf {j_0}]_{\mathfrak R}}\ar@/^1pc/[d]^{{\zeta_{\mathbf{j_0}(\ell(w_{ji}))}}}
&{{\mathcal X}_{[\mathbf i]_{\mathfrak R}}}\ar@/^1pc/[d]^{{{\mathcal T}_{w_{ji}}(\mathbf i)}}
\ar@/_1pc/[l]^{{\widehat\mu}_{\mathbf{j_0}\to{\mathbf i}}}\\
{{\mathcal X}_{[\mathbf {j_0^{\boxtimes}}]_{\mathfrak R}}}
&{{\mathcal X}_{[\mathbf i]_{\mathfrak R}}}\ar@/^1pc/[l]_{\widehat{\mu}_{\mathbf i\to\mathbf{j_0^{\boxtimes}}}}
}
\end{array}
\end{equation}
Finally, for every $v\in W$ and $\mathbf i,\mathbf j,
\mathbf k\in D(v)$, the transitive equality
$\widehat{\mu}_{\mathbf i\to\mathbf k}=\widehat{\mu}_{\mathbf j\to\mathbf k}
\circ\widehat{\mu}_{\mathbf i\to\mathbf j}$ implies the commutativity of the
following diagrams, where $\Id$ denotes the identity map on ${\mathcal X}_{[\mathbf i]_{\mathfrak R}}$.

\begin{equation}\label{equ:synth3}\xymatrix{
{\mathcal X}_{[\mathbf i]_{\mathfrak R}}
\ar@/^1pc/[r]_{\widehat{\mu}_{\mathbf{i}\to\mathbf {i_0}}}\ar@/^3pc/[rrr]^{\Id}
&{{\mathcal X}_{[\mathbf {i_0}]_{\mathfrak R}}}
\ar@/^1pc/[r]^{\widehat{\mu}_{\mathbf{i_0}\to\mathbf {j_0}}}
&{{\mathcal X}_{[\mathbf {j_0}]_{\mathfrak R}}}
&{{\mathcal X}_{[\mathbf i]_{\mathfrak R}}}
\ar@/_1pc/[l]^{{\widehat\mu}_{\mathbf{j_0}\to{\mathbf i}}}
}
\xymatrix{
{{\mathcal X}_{[\mathbf i]_{\mathfrak R}}}\ar@/^1pc/[r]_{{\widehat\mu}_{\mathbf i\to{\mathbf{i_0^{\boxtimes}}}}}
&{{\mathcal X}_{[\mathbf {i_0^{\boxtimes}}]_{\mathfrak R}}}
&{{\mathcal X}_{[\mathbf {j_0^{\boxtimes}}]_{\mathfrak R}}}
\ar@/_1pc/[l]_{\widehat{\mu}_{\mathbf{i_0^{\boxtimes}}
\to{\mathbf{j_0^{\boxtimes}}}}}
&{{\mathcal X}_{[\mathbf i]_{\mathfrak R}}}\ar@/_1pc/[l]^{\widehat{\mu}_{\mathbf i\to\mathbf{j_0^{\boxtimes}}}}\ar@/_3pc/[lll]_{\Id}
}
\end{equation}
We now incorporate the previous diagram in the synthesis
diagram (\ref{equ:syn}), via the following three steps procedure:
$1)$ put the diagram (\ref{equ:synth1}) between the two diagrams constituting  (\ref{equ:synth2})
and identify the arrows which have the same transformation as labeling;
$2)$ put the left diagram of (\ref{equ:synth3}) upside the new diagram thus obtained and identify
the arrows which have the same transformation as labeling again;
$3)$ put the right diagram of (\ref{equ:synth3}) downside the diagram, just as before and still
identify the arrows which have the same transformation as labeling.
The relation $\mathcal{T}_{w_{ij}}(\mathbf i)=\mathcal{T}_{w_{ji}}(\mathbf i)$ is then given by
the boundary of the diagram ~(\ref{equ:syn}).
\begin{equation}
\label{equ:syn}
\xymatrix{
{\mathcal X}_{[\mathbf i]_{\mathfrak R}}\ar@/_1pc/[d]_{{{\mathcal T}_{w_{ij}}(\mathbf i)}}
\ar@/^1pc/[r]_{\widehat{\mu}_{\mathbf{i}\to\mathbf {i_0}}}\ar@/^3pc/[rrr]^{\Id}
&{{\mathcal X}_{[\mathbf {i_0}]_{\mathfrak R}}}\ar@/_1pc/[d]_{{\zeta_{\mathbf{i_0}(\ell(w_{ij}))}}}
\ar@/^1pc/[r]^{\widehat{\mu}_{\mathbf{i_0}\to\mathbf {j_0}}}
&{{\mathcal X}_{[\mathbf {j_0}]_{\mathfrak R}}}\ar@/^1pc/[d]^{{\zeta_{\mathbf{j_0}(\ell(w_{ji}))}}}
&{{\mathcal X}_{[\mathbf i]_{\mathfrak R}}}\ar@/^1pc/[d]^{{{\mathcal T}_{w_{ji}}(\mathbf i)}}
\ar@/_1pc/[l]^{{\widehat\mu}_{\mathbf{j_0}\to{\mathbf i}}}\\
{{\mathcal X}_{[\mathbf i]_{\mathfrak R}}}\ar@/_1pc/[r]^{{\widehat\mu}_{\mathbf i\to{\mathbf{i_0^{\boxtimes}}}}}
&{{\mathcal X}_{[\mathbf {i_0^{\boxtimes}}]_{\mathfrak R}}}
&{\mathcal X}_{[\mathbf {j_0^{\boxtimes}}]_{\mathfrak R}}\ar@/^1pc/[l]^{\widehat{\mu}_{\mathbf{j_0^{\boxtimes}}\to{\mathbf{i_0^{\boxtimes}}}}}
&{{\mathcal X}_{[\mathbf i]_{\mathfrak R}}}\ar@/^1pc/[l]_{\widehat{\mu}_{\mathbf i\to\mathbf{j_0^{\boxtimes}}}}\ar@/^3pc/[lll]^{\Id}
}
\end{equation}
\end{proof}

\subsection{The De-Concini-Kac-Procesi automorphisms}
Following \cite{dCKP} and \cite{B}, we define for every $j\in [1,\ell]$ and
$U_j=\exp(\mathfrak{g}_{\alpha_j})$ the homomorphism
$\xi_{j}:N_-\to U_j$
such that if $n_-\in N_-$ is factorized as a product of $u_{\beta}\in U_{\beta}$
for every negative root $\beta$ (each $\beta$ appearing only once),
then $\xi_{j}(n_-):=u_{\alpha_j}$.
Let us  recall that, as a set, the dual Poisson-Lie group $(G^*,\pi_{G^*})$ is given by
elements $(n_+t,n_-t^{-1})$ such that $n_{\pm}\in N_{\pm}$ and $t\in H$ and is isomorphic
to the set of elements $n_+t^2n_-^{-1}\in G$, where we have still $n_{\pm}\in N_{\pm}$
and $t\in H$. We denote $G^0$ this set; we therefore get a Poisson isomorphism
$(G^*,\pi_{G^*})\simeq(G^0,\pi_*)$.
For every $j\in[1,\ell]$, we denote $b^j_-:=\xi_{j}(n_-)^{-1}$ and
recall the De-Concini-Kac-Procesi automorphism
\begin{equation}\label{equ:DCKP}T_j:G^0\rightarrow G^0:n_+t^2n_-^{-1} \mapsto
\widehat{s_j}\ b^j_-\ n_+t^2n_-^{-1}\ {(\widehat{s_j}\ b^j_-)}^{-1}\ .
\end{equation}
The cluster combinatorics of the De Concini-Kac-Procesi Poisson automorphisms
$T_j$ is established via the following results.

\begin{prop}\label{lemma:tresse}Let $\mathbf{i}$ be a double word, such that
$\mathbf i=\mathbf{i_1}\mathbf{i_2}$ with $\mathbf{i_1},\mathbf{i_2}
\in R(1,w_0)$, starting with the letter ${j}$,
then $T_j=\widehat{\ev}_{\mathfrak{L}(\mathbf {i})}\circ{\mu_{\binom j{0}}}
\circ{{\widehat\ev}^{-1}_{\mathbf{i}}}$.
\end{prop}

\begin{cor}\label{lemma:tresse2}
We have the equality $T_j=\widehat{\ev}_{\mathbf {i}}\circ{\mathcal T}_j({\mathbf i})
\circ{{\widehat\ev}^{-1}_{\mathbf{i}}}$ for every $\mathbf{i}\in D(w_0)$.
\end{cor}

Corollary \ref{lemma:tresse2} is deduced from Theorem \ref{thm:ev*}, the formula
(\ref{equ:defti}) and Proposition \ref{lemma:tresse}. To prove Proposition
\ref{lemma:tresse}, we need some preparation.

Let $u,v\in W$ and $\mathbf i\in R(u,v)$. The double reduced word
$\mathbf{i^{\star}}\in R({v^\star},{u^\star})$ is obtained by transforming
each letter $i$ of $[1,\ell]\cup[\overline{1},\overline{\ell}]$ into $\overline{i^\star}$,
so if $\mathbf i=i_1\dots i_n$ then $\mathbf i^{\star}=\overline{i_1}^{\star}
\dots\overline{i_n}^{\star}$. Starting with an elementary double word
$\mathbf i\in\{\mathbf 1,i,\overline i\}$, where $i\in[1,\ell]$, and then applying the
properties of the amalgamated product, we easily prove the following lemma.

\begin{lemma}\label{lemma:w_0} Let $u,v\in W$ and $\mathbf i\in R(u,v)$.
For every cluster $\mathbf x\in{\mathcal X}_{\mathbf i}$, let $\mathbf{x^{\star}}
\in{\mathcal X}_{\mathbf{i^\star}}$ be such that the equality
$\widehat{w_0}\ev_{\mathbf i}(\mathbf x)\widehat{w_0}^{-1}=\ev_{\mathbf{i^\star}}
(\mathbf{x^{\star}})$ is satisfied.
Then we have
\begin{equation}\label{equ:w_0}x^{\star}_{\binom i{j}}=\left\{
\begin{array}{rl}
-x_{\binom {i^\star}{j}}^{-1}&\mbox{if } 0=j\neq N^{i^\star}(\mathbf i)
\mbox{ or } 0\neq j= N^{i^\star}(\mathbf i);\\
x_{\binom {i^\star}{j}}^{-1}&\mbox{otherwise}.
\end{array}
\right.
\end{equation}
\end{lemma}

A \emph{split} of a seed $\mathbf I$ is a pair
of seeds $(\mathbf I_1,\mathbf I_2)$ such that $\mathbf I$ is their
amalgamated product, that is $\mathbf I=\mathfrak{m}(\mathbf I_1,\mathbf I_2)$.
An associated \emph{$\mathcal X$-split}
is a section of the amalgamation map $\mathfrak{m}:{\mathcal X}_{\mathbf I_1}
\times{\mathcal X}_{\mathbf I_2}\rightarrow{\mathcal X}_{\mathbf I}$, i.e.
a map $\mathfrak{s}:{\mathcal X}_{\mathbf I}\rightarrow{\mathcal X}_{\mathbf I_1}
\times{\mathcal X}_{\mathbf I_2}$ such that the product $\mathfrak{m}\circ\mathfrak{s}$
gives the identity map on ${\mathcal X}_{\mathbf I}$.
For every $\mathcal X$-split $\mathfrak{s}$ associated to the decomposition
$\mathbf I\to(\mathbf I_1,\mathbf I_2)$, we associate to any
$\mathbf x\in{\mathcal X}_{\mathbf I}$, some elements $\mathbf x_{(1)}
\in{\mathcal X}_{\mathbf{I_1}}$ and $\mathbf x_{(2)}\in{\mathcal X}_{\mathbf{I_2}}$
given by $\mathfrak{s}(\mathbf{x})=(\mathbf x_{(1)},\mathbf x_{(2)})$.

Now, for every reduced word $\mathbf i=\overline i_1\dots\overline i_{\ell(w_0)}\in R(w_0,1)$,
and every $k\in[1,\ell(w_0)]$, let us set $w_{\mathbf i_{>k}}:=s_{i_{k+1}}\dots s_{i_{\ell(w_0)}}$.
To every $\mathbf x\in{\mathcal X}_{\mathbf i}$, we associate the following product of $u_{\beta}\in U_{\beta}$
over negative roots, which is such that every negative root $\beta$ appears exactly once.
$$\begin{array}{ccc}
\tau_{\mathbf i}(\mathbf x)=\displaystyle\prod_{k=1}^{\ell(w_0)}
{\widehat{w}_{\mathbf i_{>k}}}^{-1}x_{\overline{i_k}}(-x_{\zeta_{\leq k-1}\binom{i_k}{0}}^{-1})\widehat{w}_{\mathbf i_{>k}}\ ,
&\mbox{where}&x_{\overline{i}} (t)=\varphi_i\left(
\begin{array}{cc}
1 & 0\\
t & 1
\end{array}
\right).
\end{array}$$
Here, the map $\zeta_{\leq k-1}:{\mathcal X}_{\mathbf i}\to{\mathcal X}_{\mathbf i(k-1)}$
is obtained by generalizing the formula (\ref{equpartialzeta}):
$$\zeta_{\leq k-1}=\zeta_{\mathbf i(1)}\circ\dots\circ\zeta_{\mathbf i(k-1)}\ .$$
In particular, the equality $\zeta_{\leq \ell(w_0)}=\zeta_{\mathbf i}$ is satisfied.

\begin{lemma}\cite[Lemma 9.10]{RB}\label{cor:beta} Let $\mathbf{i}$ be a double word such
that $\mathbf i=\mathbf{i_1}\mathbf{i_2}$ with $\mathbf{i_1},\mathbf{i_2}\in R(1,w_0)$,
and $\mathfrak{s}$ be a $\mathcal X$-split relative to the decomposition $\mathbf i\to(\mathbf{i_1},\mathbf{i_2})$.
We have the equality $[[\widehat{\ev}_{\mathbf{i}}(\mathbf x)]]_-^{-1}=\tau_{{\mathbf{i_1^{\star}}}}
(\mathbf{x}_{(1)}^{\star})$.
\end{lemma}

\begin{rem}The definition of $\tau_{{\mathbf{i_1^{\star}}}}$ implies that
the choice of the $\mathcal X$-split $\mathfrak{s}$ associated to the decomposition
$\mathbf i\to(\mathbf i_1,\mathbf i_2)$ in the previous lemma doesn't matter.
\end{rem}

We can now prove Proposition \ref{lemma:tresse}.

\begin{proof}Let $\mathfrak{s}$ be a $\mathcal X$-split relative to the decomposition
$\mathbf i\to(\mathbf{i_1},\mathbf{i_2})$. Lemma \ref{cor:beta} gives an expression of
$[[\widehat{\ev}_{\mathbf i}(\mathbf x)]]_-^{-1}$ as a product of terms $u_{\beta}$
where all the negative roots $\beta$ appear exactly once. Now, let us remark
that~$w_0s_{i_1^{\star}}(\alpha_{i_1^{\star}})=\alpha_{i_1}$. Applying the definition
of $\xi_j$, Lemma \ref{cor:beta} and the formula (\ref{equ:w_0}), we thus get
$$
{b_-^j}=\xi_{j}([[\widehat{\ev}_{\mathbf i}(\mathbf x)]]_-^{-1})^{-1}
=x_{\overline{j}}(-{x^{\star}_{\binom{j^{\star}}{0}}}^{-1})^{-1}
=x_{\overline{j}}({x_{\binom{j}{0}}})^{-1}
=x_{\overline{j}}(-{x_{\binom{j}{0}}})\ .
$$
Define $\mathbf{j_1}$
and $\mathbf y\in{\mathcal X}_{\mathbf{j_1}}$ such that $\mathbf{j_1}={\mathfrak L}(\mathbf{i_1})$
and $\mathbf y=\mu_{\binom{i_1}{0}}(\mathbf{x}_{(1)})$. We have
\begin{equation}\label{equ:sitrop}{\widehat{s_{j}}}^{-1}\ev_{\mathbf{j_1}}(\mathbf {y})
=x_{\overline{j}}(-y^{-1}_{\binom{j}{0}})\ \ \ev_{{\mathfrak L}(\mathbf{j_1})}
\circ\mu_{\binom {j}{0}}(\mathbf {y})\ .
\end{equation}
Indeed, let us remember the map $\varphi_j:\SL(2,\mathbb C)\hookrightarrow G$
defined in Section~\ref{section:prelimanaries}. For any nonzero $t \in \mathbb C$
and any $i\in [1,\ell]$, let us denote
$$
\begin{array}{cccccc}
x_{i} (t) =\varphi_i\left(
\begin{array}{cc}
1 & t\\
0 & 1
\end{array}
\right)
&,&
x_{\overline{i}} (t)=\varphi_i\left(
\begin{array}{cc}
1 & 0\\
t & 1
\end{array}
\right).
\end{array}
$$
An elementary matrix calculus on $\SL(2,\mathbb C)$ leads to the following
equality on $G$, satisfied for every $j\in[1,\ell]$.
$${\widehat{s_j}}^{-1}x_{\overline{j}}(t)
=x_{\overline{j}}(-t^{-1})\ \exp(\log(t){h_j})\ x_j(t^{-1})\ .
$$
Using the definition of tropical mutation, and the fact that a left tropical
mutation commutes with an amalgamation done on the right, we deduce the relation
(\ref{equ:sitrop}). Then, from (\ref{equ:sitrop}), we get the following series of
equalities
$$\begin{array}{lll}
\widehat{s_j}\ {b_-^j}\ev_{\mathbf{i_1}}(\mathbf{x}_{(1)})
&=\widehat{s_j}\ x_{\overline{j}}(-{x_{\binom{j}{0}}})\ev_{\mathbf{i_1}}(\mathbf{x}_{(1)})
=\widehat{s_j}\ x_{\overline{i_1}}(-{y_{\binom{i_1}{0}}^{-1}})
\ev_{{\mathfrak L}(\mathbf{j_1})}(\mu_{\binom j{0}}(\mathbf{y}))\\
&=\widehat{s_j}\ \widehat{s_j}^{-1}\ev_{\mathbf{j_1}}(\mathbf{y})
=\ev_{\mathbf{j_1}}(\mathbf{y})
=\ev_{{\mathfrak L}(\mathbf{i_1})}\circ{\mu_{\binom j{0}}}(\mathbf{x}_{(1)})\ .
\end{array}
$$
This relation is then extended to the  evaluation maps $\ev_{\mathbf i}^{\mathfrak L}$,
$\ev_{\mathbf i}^{\mathfrak R}$ and $\widehat{\ev}_{\mathbf i}$
by using (\ref{equ:evhat}) and the fact that a left tropical
mutation commutes with an amalgamation done on the right. Thus, the definition (\ref{equ:DCKP})
leads to the equality $T_j\circ{{\widehat\ev}_{\mathbf{i}}}=\widehat{\ev}_{\mathfrak{L}
(\mathbf {i})}\circ{\mu_{\binom j{0}}}$.
\end{proof}

\section{The case $G=\PGL(2,\mathbb C)$.}\label{section:SL2}
To fix the ideas, we consider our construction in the case $\PGL(2,\mathbb C)$.
Let us recall that the complex simple Lie group
\begin{equation}\label{equ:sl2}
SL(2, \mathbb{C})=\{\left(
\begin{array}{cc}
t_{11} & t_{12}\\
t_{21} & t_{22}
\end{array}
\right):t_{11}t_{22}-t_{12}t_{21}=1,\ \  t_{ij}\in \mathbb{C}\}\ .
\end{equation}
has its Lie algebra $\mathfrak{g}$ equal to the vector space $\sl(2,\mathbb C)$ of
$2$-squared complex matrices which have a zero trace. The Chevalley generators
$\{e_1,f_1,h_1\}$ and its related basis $\{e_1,f_1,h^1\}$ are then given by
the following matrices:
$$\begin{array}{llll}
e_1=\left(
\begin{array}{cc}
0 & 1\\
0 & 0
\end{array}
\right), &
f_1=\left(
\begin{array}{cc}
0 & 0\\
1 & 0
\end{array}
\right), &
h_1=\left(
\begin{array}{cc}
1 & 0\\
0 & -1
\end{array}
\right), &
h^1=\left(
\begin{array}{cc}
1/2 & 0\\
0 & -1/2
\end{array}
\right)
\end{array}\ .$$
Using the exponential map $\exp:\mathfrak{g}\rightarrow G$, which, in this
case, associates to a matrix $M\in\mathfrak g$ the usual matrix
$\sum_{n=0}^{\infty}\frac{M^n}{n!}\in G$,
we get the following generators of $G$, the two last ones being associated
to every non-zero complex number $x$.
$$\begin{array}{llll}
E^1=\left(
\begin{array}{cc}
1 & 1\\
0 & 1
\end{array}
\right), &
F^1=\left(
\begin{array}{cc}
1 & 0\\
1 & 1
\end{array}
\right), &
H_1(x)=\left(
\begin{array}{cc}
x & 0\\
0 & x^{-1}
\end{array}
\right),&
H^1(x)=\left(
\begin{array}{cc}
x^{1/2} & 0\\
0 & x^{-1/2}
\end{array}
\right)
\end{array}\ .$$
Let us notice that $H^1(x)$ is well-defined on $\PGL(2,\mathbb C)$, because of the identity
$$H^1(x)=\left(
\begin{array}{cc}
x^{1/2} & 0\\
0 & x^{-1/2}
\end{array}
\right)\stackrel{\PGL(2,\mathbb C)}{=}
\left(
\begin{array}{cc}
x & 0\\
0 & 1
\end{array}
\right)\ .$$
Now, because there is only one simple root $\alpha_1$, the Weyl group $W$ contains only two elements
$\{1,s_1\}$ and the different double reduced words are the double words $\mathbf 1$, $1$,
$\overline{1}$, $1\overline{1}$, $\overline{1}1$, where $\mathbf 1$ is the unity of the direct product
$W\times W$. Finally, the $r$-matrix $r\in\mathfrak g\wedge\mathfrak g$ associated
to $\sl(2,\mathbb C)$ and its related elements $r_{\pm}\in\mathfrak g\otimes\mathfrak g$
are given by the following formulas.
\begin{equation}\label{equ:elemrmatrix}
\begin{array}{cccc}
r=e_1\wedge f_1,&r_+=\displaystyle\frac{1}{4}h_1\otimes h_1+e_1\otimes f_1&\mbox{and}&
r_-=-\displaystyle\frac{1}{4}h_1\otimes h_1-f_1\otimes e_1\ .
\end{array}
\end{equation}
For every $i,j\in\{1,2\}$, let $t_{ij}$ be the coordinate function on the matrices
(\ref{equ:sl2}). Applying the formula (\ref{equ:elemrmatrix}) on the
Semenov-Tian-Shansky Poisson bracket given by Proposition \ref{prop:STSPoisson},
it is easy to prove that in the matricial case, the Poisson bracket on $(G,\pi_*)$ is
given by the following equalities:
$$\left\{\begin{array}{lll}
\{t_{11},t_{12}\}_*=t_{12}t_{22}, & \{t_{11},t_{21}\}_*=-t_{21}t_{22}\ , \\
\{t_{11},t_{22}\}_*=0, & \{t_{12},t_{21}\}_*=t_{11}t_{22}-t_{22}^2\ , \\
\{t_{12},t_{22}\}_*=t_{12}t_{22}, & \{t_{21},t_{22}\}_*=-t_{21}t_{22}\ .
\end{array}
\right.$$
Let us then consider the related evaluation maps. It is
easy to check that the evaluation~$\widehat{\ev}_1:{\mathcal X}_{[1]_{\mathfrak R}}
\to (G,\pi_*)$ is Poisson. It is indeed given by the following expression:

$$\begin{array}{rl}
\widehat{\ev}_1(x_0,t)
=\left(\begin{array}{cc}
t^{1/2}+t^{-1/2} & -x_0t^{-1/2}\\
x_0^{-1}t^{1/2} & 0
\end{array}\right) .
\end{array}$$
The evaluations $\widehat{\ev}_{\overline 11}:{\mathcal X}_{[\overline 11]_{\mathfrak R}}\to BB_-$
and $\widehat{\ev}_{{1}\overline 1}:{\mathcal X}_{[1\overline 1]_{\mathfrak R}}\to BB_-$,
parameterizing the subvariety~$BB_-$, are then obtained by the following formulas.

$$\begin{array}{rl}
\widehat{\ev}_{\overline{1}1}(y_0,y_1,t)=\left(\begin{array}{cc}
t^{-1/2}(1+y_1)+t^{1/2} & -y_0y_1t^{-1/2}\\
y_0^{-1}(t^{1/2}(1+y_1^{-1})+t^{-1/2}(1+y_1)) & -y_1t^{-1/2}
\end{array}\right)\ .\\
\\
\widehat{\ev}_{1\overline{1}}(\widetilde{y_0},\widetilde{y_1},t)=\left(\begin{array}{cc}
t^{-1/2}(1+\widetilde{y_1}^{-1})+t^{1/2} & -t^{-1/2}\widetilde{y_0}(1+\widetilde{y_1}^{-1})\\
\widetilde{y_0}^{-1}(t^{1/2}+t^{-1/2}\widetilde{y_1}^{-1}) & -\widetilde{y_1}^{-1}t^{1/2}
\end{array}\right).
\end{array}$$
And it is straightforward to check  that $\mu_{[\overline 1{1}]_{\mathfrak{R}}
\to[{1}\overline 1]_{\mathfrak{R}}}:(y_0,y_1,t)\mapsto(\widetilde{y_0},
\widetilde{y_1},t)$.
The remaining twisted
evaluations $\widehat{\ev}_{11},\widehat{\ev}_{1\overline 1}:
{\mathcal X}_{[11]_{\mathfrak R}}\to BB_-$ and
$\widehat{\ev}_{\overline{1}\overline 1}:
{\mathcal X}_{[\overline 11]_{\mathfrak R}}\to BB_-$
are given by:

$$\begin{array}{ll}
\widehat{\ev}_{11}(z_0,z_1,t)=\widehat{\ev}_{1\overline 1}(z_0,z_1,t)
=\left(\begin{array}{cc}
(1+z_1^{-1})t^{1/2}+t^{-1/2} & -z_0((1+z_1^{-1})t^{1/2}+(1+z_1)t^{-1/2})\\
z_0^{-1}z_1^{-1}t^{1/2} & -z_1^{-1}t^{1/2}
\end{array}\right)\ .\\
\\
\widehat{\ev}_{\overline{1}\overline 1}(y_0,y_1,t)=\widehat{\ev}_{\overline{1} 1}(y_0,y_1,t)
=\left(\begin{array}{cc}
t^{-1/2}(1+y_1)+t^{1/2} & -t^{-1/2}y_0y_1\\
y_0^{-1}(t^{1/2}(1+y_1^{-1})+t^{-1/2}(1+y_1)) & -y_1t^{-1/2}
\end{array}\right)\ .
\end{array}$$
It is easy to check that all these maps are Poisson when the
matrices establishing the Poisson structure
on the related seed ${\mathcal X}$-tori are given by
$$
\begin{array}{ccc}
\eta({1{1}})=\eta({1\overline{1}})=\left(
\begin{array}{ccc}
0 & -1&0\\
1 & 0&0\\
0&0&0
\end{array}
\right),&
\eta({\overline 1 1})=\eta({\overline 1 \overline 1})=\left(
\begin{array}{ccc}
0 & 1&0\\
-1 & 0&0\\
0&0&0
\end{array}
\right).
\end{array}
$$
We thus get two
cluster $\mathcal X$-varieties for the variety $BB_-$, denoted ${\mathcal X}_{e}$
and ${\mathcal X}_{w_0}$, and respectively associated to the cluster variables
$(y_0,y_1,t)$ and $(z_0,z_1,t)$. They are linked in the following way: if the
elements $\widehat{\ev}_{\overline{1}1}(y_0,y_1,t)$ and
$\widehat{\ev}_{11}(z_0,z_1,t)$  are equal, we quickly check
with the expressions above that the map $\varphi:(y_0,y_1,t)
\mapsto (z_0,z_1,t)$ is given by
$$
\left\{\begin{array}{rcl}
z_0&=&y_0{(1+y_1^{-1})}^{-1}{(1+y_1^{-1}t)}^{-1}\\
z_1&=&ty_1^{-1}
\end{array}\right. .
$$
In fact, we have the equality $\varphi=\Xi_{s_1}$,
coming from the following formula for~$\Xi_{s_1}$.
$$
\begin{array}{rll}
\Xi_{s_1}(y_0,y_1,t)&=\mu_{[{1}\overline 1]_{\mathfrak R}\to
[\overline 1{1}]_{\mathfrak R}}\circ\Xi_1\circ\mu_{[\overline{1}1]_{\mathfrak R}
\to [1\overline{1}]_{\mathfrak R}}(y_0,y_1,t)\\
\\
&=(y_0{(1+y_1^{-1})}^{-1}{(1+y_1^{-1}t)}^{-1},y_1^{-1}t,t)\ .
\end{array}
$$
Finally, we use tropical mutations to describe the De-Concini-Kac-Procesi
Artin group action (\ref{equ:defti}) on $(\PGL(2,\mathbb C),\pi_*)$. The cluster
combinatorics is given by the following Poisson automorphism on
the seed $\mathcal X$-torus ${\mathcal X}_{[1 1]_{\mathfrak R}}$.
$$\begin{array}{ccll}
{\mathcal T}_1({1 1}):&{\mathcal X}_{[1 1]_{\mathfrak R}}\longrightarrow
{\mathcal X}_{[1 1]_{\mathfrak R}}:&(z_0,z_1,t)\longmapsto\Xi_{s_1}\circ\mu_{\binom{1}{0}}(z_0,z_1,t)
\end{array}$$
$$\begin{array}{ccll}
\mbox{satisfies}&
{\mathcal T}_1({1 1})(z_0,z_1,t)=(z_0^{-1}{(1+z_1^{-1})}^{-1}{(1+z_1^{-1}t)}^{-1},z_1^{-1}t,t)\ .
\end{array}
$$
Let us stress, however, that this birational Poisson isomorphism is not an involution; indeed,
a straightforward computation gives the equality
${\mathcal T}_1({11})^2(z_0,z_1,t)=(z_0z_1^{-2}t^2,z_1,t)$.
Therefore, the action of the center ${\mathcal Z}({\mathcal B}_{\mathfrak g})$ of
${\mathcal B}_{\mathfrak g}$ on $(BB_-,\pi_*)$
given by the De Concini-Kac-Procesi automorphism is not trivial.

\end{document}